\documentclass[reqno]{amsart}
\usepackage{graphicx}
\usepackage[normalem]{ulem}
\usepackage[foot]{amsaddr}
\usepackage[a-1b]{pdfx}
\usepackage[utf8]{inputenc}
\usepackage[T1]{fontenc}
\usepackage{amsfonts}
\usepackage{amssymb}
\usepackage{amsthm}
\usepackage{mathtools}
\mathtoolsset{mathic=true}
\usepackage{comment}
\usepackage{esint}
\usepackage{hyperref}
\usepackage{xcolor}
\hypersetup{
  final,
   colorlinks,
   linkcolor={red!70!black},
   citecolor={blue!70!black},
   urlcolor={blue!90!black}
}
\usepackage{microtype}
\usepackage{dutchcal}
\usepackage{upgreek}
\usepackage{fourier}

\usepackage{pgfornament}
\newcommand\HERE{%
  \par\bigskip\noindent\hfill \hfill\pgfornament[width=.1\textwidth]{7}\hfill\pgfornament[width=.1\textwidth]{7}\hfill\pgfornament[width=.1\textwidth]{7}\hfill\hfill\null\par\bigskip
}

\newcommand{\HH}{\mathbb{H}}

\newcommand{\R}{\mathbb{R}}

\newcommand{\cF}{\mathcal{F}}

\newcommand{\cS}{\mathcal{S}}

\renewcommand{\epsilon}{\varepsilon}

\renewcommand{\approx}{\asymp}

\newcommand{\ud}[0]{\,\mathrm{d}}
\newcommand{\ovln}[1]{#1}

\DeclareMathOperator{\Exc}{Exc}
\DeclareMathOperator{\BV}{BV}
\DeclareMathOperator{\spanop}{span}
\newcommand{\rstr}{\ensuremath{\lfloor}}

\makeatletter
\def\resetMathstrut@{%
  \setbox\z@\hbox{%
    \mathchardef\@tempa\mathcode`\(\relax
    \def\@tempb##1"##2##3{\the\textfont"##3\char"}%
    \expandafter\@tempb\meaning\@tempa \relax
  }%
  \ht\Mathstrutbox@1.2\ht\z@ \dp\Mathstrutbox@1.2\dp\z@
}
\makeatother

\newcommand{\one}{\mathbf{1}}

\newcommand{\from}{\colon}
\newcommand{\symdiff}{\mathbin{\triangle}}
\renewcommand{\mid}{:}
\newcommand{\eqdef}{\stackrel{\mathrm{def}}{=}}

\renewcommand{\ge}{\geqslant}

\renewcommand{\le}{\leqslant}
\renewcommand{\setminus}{\smallsetminus}

\DeclareMathOperator{\sign}{sign}
\DeclareMathOperator{\supp}{supp}
\DeclareMathOperator{\id}{id}
\DeclareMathOperator{\Per}{Per}
\DeclareMathOperator{\Lip}{Lip}
\let\div\relax
\DeclareMathOperator{\div}{div}

\newtheorem{thm}{Theorem}[section]

\newtheorem{lemma}[thm]{Lemma}
\newtheorem{prop}[thm]{Proposition}
\newtheorem{cor}[thm]{Corollary}
\theoremstyle{remark}
\newtheorem{defn}[thm]{Definition}
\newtheorem{remark}[thm]{Remark}
\newtheorem{example}[thm]{Example}
\newtheorem{question}[thm]{Question}

\newcommand{\pd}[2]{\partial_{#2} #1}

\title{Harmonic intrinsic graphs in the Heisenberg group}
\author{Robert Young}
\address{New York University, Courant Institute of Mathematical Sciences, 251 Mercer Street, New York, NY 10012, USA.}
\email{ryoung@cims.nyu.edu}

\thanks{
  This material is based upon work supported by the National Science Foundation under Grant Nos.\ 2005609 and 1926686 and research done while the author was a visiting member at the Institute of Advanced Study.   
}

\hypersetup{final}
\begin{document}
\begin{abstract}
  Minimal surfaces in $\R^n$ can be locally approximated by graphs of harmonic functions, i.e., functions that are critical points of the Dirichlet energy, but no analogous theorem is known for $H$--minimal surfaces in the three-dimensional Heisenberg group $\HH$, which are known to have singularities. In this paper, we introduce a definition of intrinsic Dirichlet energy for surfaces in $\HH$ and study the critical points of this energy, which we call contact harmonic graphs. Nearly flat regions of $H$--minimal surfaces can often be approximated by such graphs. We give a calibration condition for an intrinsic Lipschitz graph to be energy-minimizing, construct energy-minimizing graphs with a variety of singularities, and prove a first variation formula for the energy of intrinsic Lipschitz graphs and piecewise smooth intrinsic graphs. 
\end{abstract}
\maketitle

\section{Introduction}  

Minimal surfaces in $\R^3$ are smooth, but the analogous $H$--minimal surfaces in the three-dimensional Heisenberg group need not be. Recall that an $H$--minimal surface is a stationary point of the Heisenberg area functional, which is proportional to the $3$--dimensional spherical Hausdorff measure. There are many examples of $H$--minimal surfaces with a singularity along a curve \cite{PaulsHMinimal}, as shown in Figure~\ref{fig:introMinimal}. 

\begin{figure} \begin{centering}\includegraphics[width=.5\textwidth]{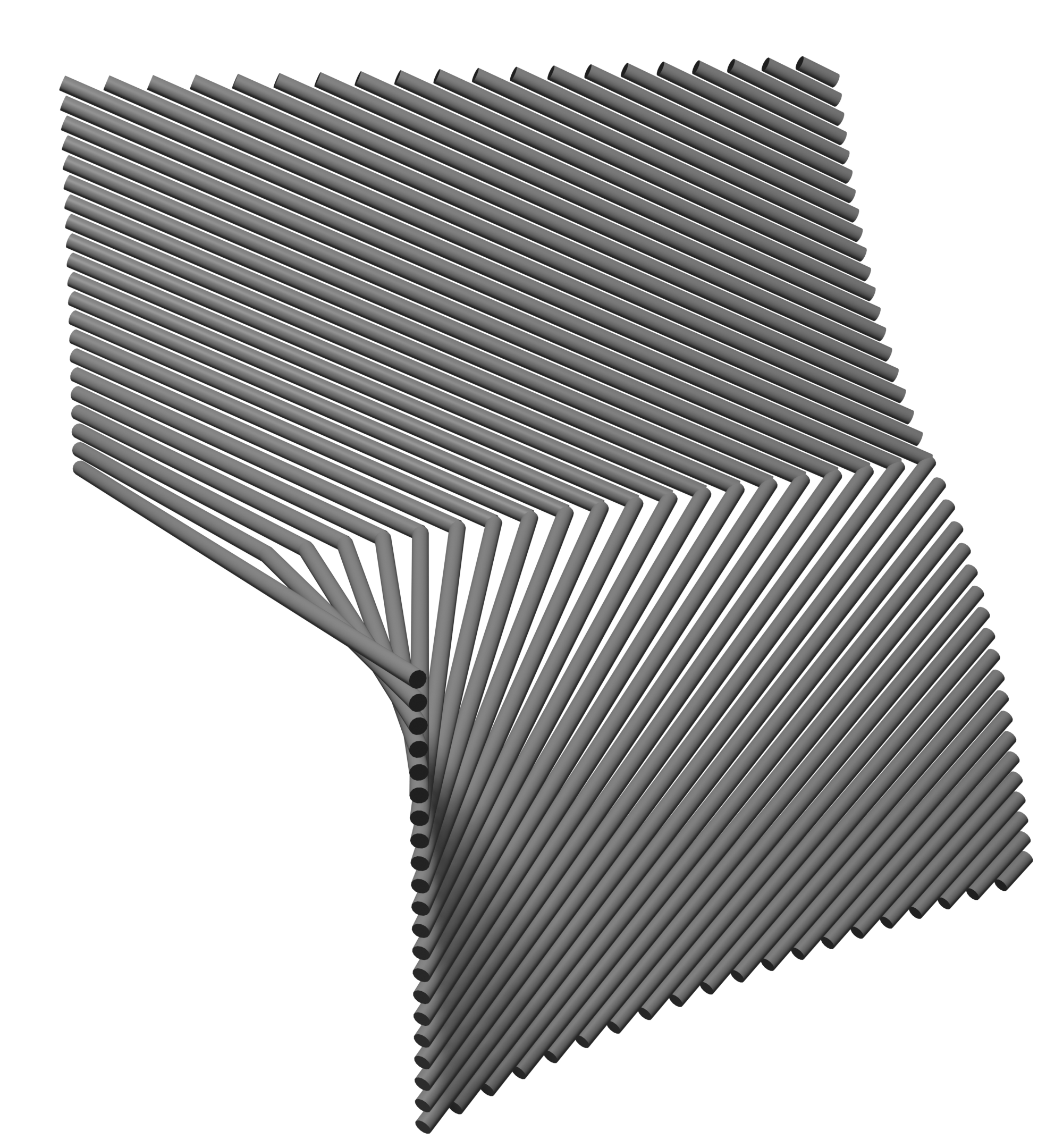}%
    \includegraphics[width=.5\textwidth]{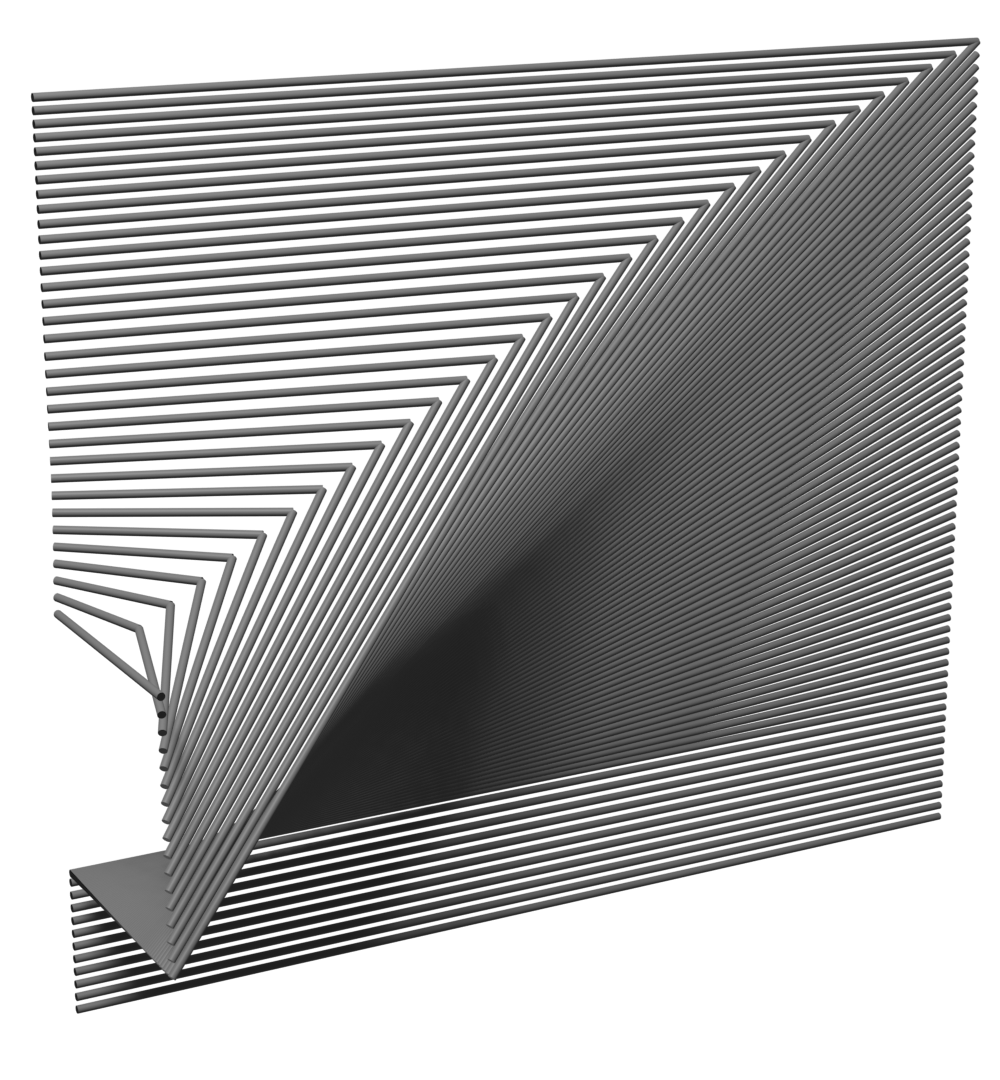}
  \end{centering}
  \caption{\label{fig:introMinimal}
    Examples of $H$--minimal and contact harmonic graphs with singularities. In both cases, the surfaces consist of horizontal line segments which intersect only along the characteristic nexus. $3$D models can be found in the supplementary materials.}
\end{figure}

\begin{figure} 
  \begin{centering}
    \includegraphics[width=.5\textwidth]{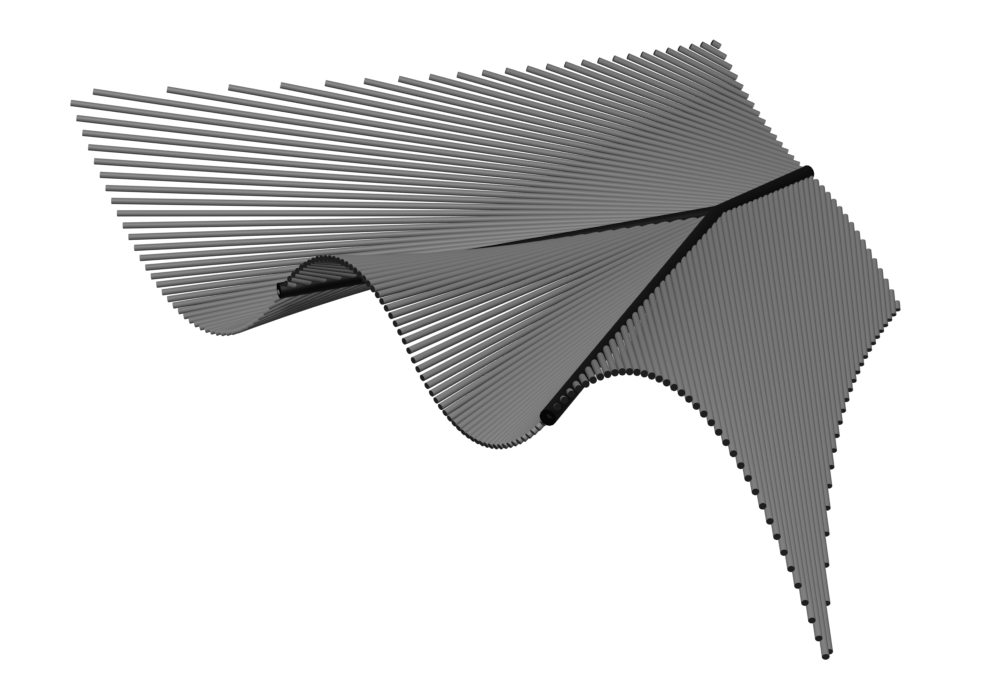}%
    \includegraphics[width=.5\textwidth]{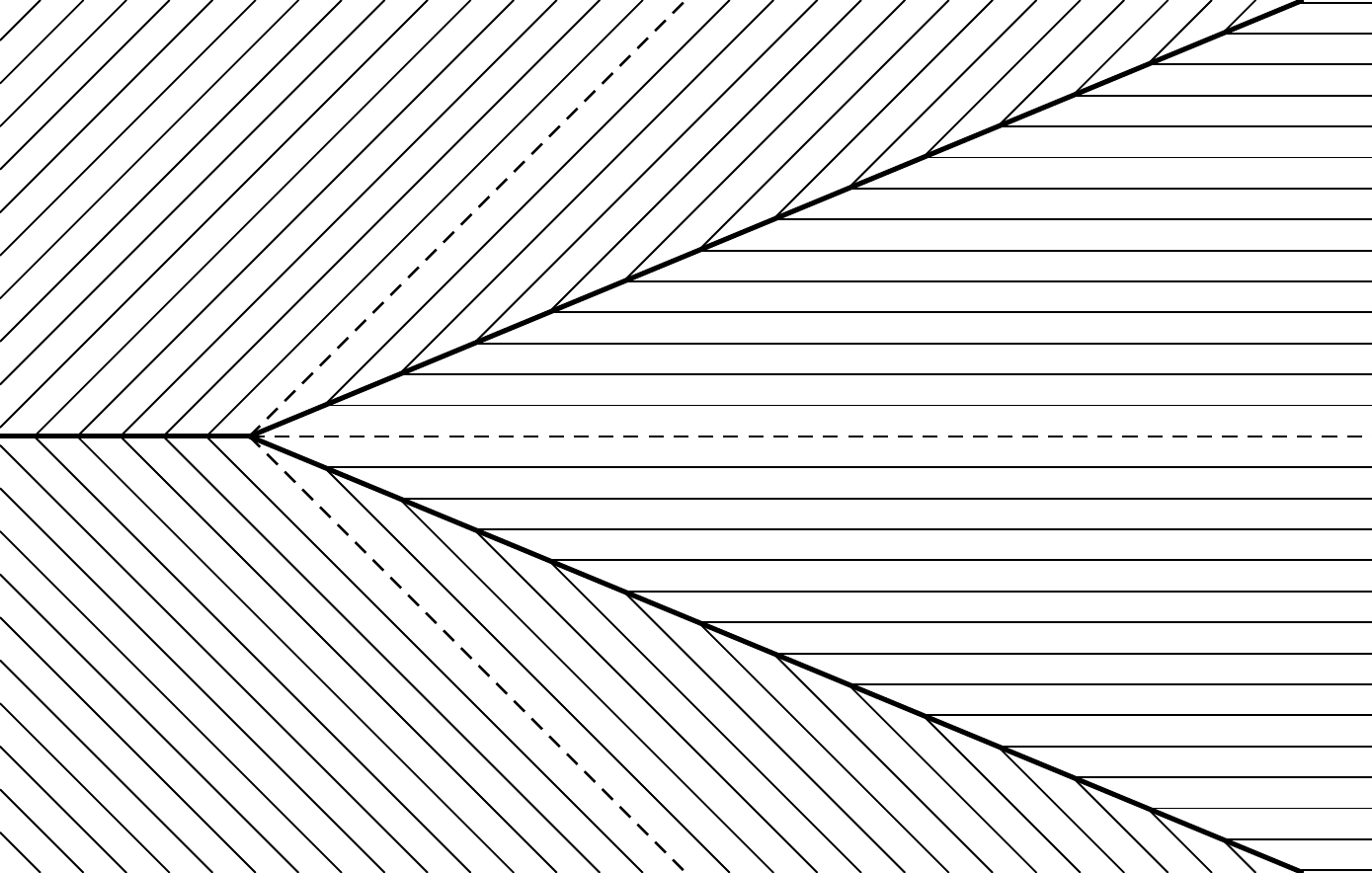}
  \end{centering}
  \caption{\label{fig:branch}
    An $H$--minimal graph with branched characteristic nexus (left) and the foliation of $\Sigma$ by horizontal curves, projected to the $xy$--plane (right). Dashed lines mark the boundaries between the pieces of herringbone surfaces (see Lemma~\ref{lem:herringbone}) that make up $\Sigma$ and thick lines mark the characteristic nexus. A $3$D model of the surface can be found in the supplementary materials.
  }
\end{figure}

In fact, $H$--minimal surfaces can have singularities in regions that are close to flat. The surfaces in Figure~\ref{fig:introMinimal}, for example, grow closer to planes as one moves from left to right. Subtler examples come from a construction of Nicolussi Golo and Ritoré \cite{GR2020}.
\begin{thm}[\cite{GR2020}]\label{thm:constructArea}
  Let $S^1$ be the unit circle of horizontal vectors in $\HH$.
  Given a collection (finite or countably infinite) of disjoint nonempty open intervals $I_1,I_2,\dots\subset S^1$ such that $\ell(I_i)<2\pi$ for all $i$, let $K=S^1\setminus \bigcup_i I_i$ be its complement.  Let $I_i=(a_i,b_i)$ and let $m_i\in S^1$ be the midpoint of $I_i$.  Let $\Sigma_K$ be the surface consisting of the union of:
  \begin{itemize}
  \item for each $k\in K$, a horizontal ray from the origin in the direction of $k$,
  \item for each $i$, a horizontal ray $R_i$ from the origin in the direction of $m_i$, and
  \item for each $p\in R_i$, a pair of horizontal rays from $p$ in the directions of $a_i$ and $b_i$.
  \end{itemize}
  Then $\Sigma_K$ is a scale-invariant area-minimizing surface.
\end{thm}
Figure~\ref{fig:branch} shows an example of such a surface; here, the rays $R_i$ are the thick lines and $K$ consists of the three points in $S^1$ corresponding to the dashed lines.

If we bring these three points close together, $\Sigma_K$ converges to a vertical plane; one can show (see Proposition~\ref{prop:stretchLimits}) that if $K=\{-\epsilon, 0, \epsilon\}$, then $\Sigma_K$ converges to the $xz$--plane as $\epsilon\to 0$. That is, $\Sigma_K$ can be arbitrarily close to the $xz$--plane in $\HH$ and still have a branched singularity.

In fact, there are examples with worse behavior. Suppose that $K$ is a Cantor set with positive measure and let $\chi=\bigcup R_i$. Since $K$ is a Cantor set, the closure $\overline{\chi}$ contains the cone over $K$, which has positive measure. The rectifiability of $\Sigma_K$ implies that almost every point in $\overline{\chi}$ has a vertical approximate tangent plane; indeed, this holds for every point in $\overline{\chi}\setminus \chi$. If $p\in \overline{\chi}\setminus \chi$, then 
$\Sigma_K$ is flat near $p$ in the sense that the neighborhoods $\Sigma_K\cap B(p,\epsilon)$ are close to a vertical plane for sufficiently small $\epsilon$, but singular in the sense that $\Sigma_K\cap B(p,\epsilon)$ is nonsmooth for all $\epsilon$ and the unit horizontal normal to $\Sigma_K\cap B(p,\epsilon)$ is discontinuous.

This does not happen in $\R^n$. A key step in proving the regularity of codimension--$1$ minimal surfaces in $\R^n$ is to show that if $E$ is a perimeter-minimizing subset of $\R^n$ and $p$ lies in the reduced boundary of $E$, then $\partial E$ can be approximated by graphs of harmonic functions on small balls around $p$. One can then use the regularity of harmonic functions to show that  $\partial E$ is smooth on sufficiently small balls around $p$; see for instance \cite{Maggi} for an exposition.

In this paper, we propose a method to study singular points in $H$--minimal surfaces by using an analogue of harmonic functions that we call \emph{contact harmonic graphs}. Many of these graphs arise as limits of $H$--minimal surfaces under rescaling and stretching, and we conjecture that, as in $\R^n$, one can use harmonic graphs to approximate sufficiently flat regions of $H$--minimal surfaces. A harmonic graph is a critical point of the \emph{intrinsic Dirichlet energy} (see \eqref{eq:ide intro}); we describe a calibration method to prove that a graph is energy-minimizing and use calibrations to construct several examples. In addition, we prove a first variation formula for the intrinsic Dirichlet energy that lets us characterize intrinsic Lipschitz harmonic graphs.

Before we state our results, we set some notation. We define the Heisenberg group $\HH$ to be the set $\R^3$ equipped with the product
$$(x,y,z)\cdot(x',y',z')=\left(x+x',y+y',z+z'+\frac{xy'-yx'}{2}\right).$$
Let $x,y,z\from \HH\to \R$ be the coordinate functions and let $X,Y, Z$ be the coordinate vectors. The one-parameter subgroups of $\HH$ generated by these vectors are the coordinate axes, so we write elements of these one-parameter subgroups as $X^t=(t,0,0)$, $Y^t=(0,t,0)$, and $Z^t=(0,0,t)$, for $t\in \R$.  Let $V_0=\{(x,y,z)\in \HH\mid y=0\}$ be the $xz$--plane.  For any $f\from V_0\to \R$, we define the intrinsic graph of $f$ to be
$$\Gamma_f=\{v\cdot Y^{f(v)}\mid v\in V_0\}$$
and define $\Psi_f\from V_0\to \Gamma_f$ by $\Psi_f(v)=v\cdot Y^{f(v)}$.  Let $\mu$ be Lebesgue measure on $V_0$.

When $f$ is an \emph{intrinsic Lipschitz function} (see Definition~\ref{def:intrinsic lipschitz graph}), $\Gamma_f$ has approximate tangent planes with respect to the Carnot--Carathéodory metric almost everywhere \cite{FSSCDiff} and these planes are vertical (parallel to the $z$--axis).  Each tangent plane projects to a line in the $xy$--plane with slope given by a nonlinear differential operator, the \emph{intrinsic gradient}.  When $f$ is smooth, the intrinsic gradient of $f$ is given by 
$$\nabla_f f = (\partial_x - f\partial_z)f;$$
when $f$ is not smooth, we define $\nabla_ff$ in the distributional sense (see Definition~\ref{def:distrib gradient}).

Given a bounded open subset $U\subset V_0$ and an intrinsic graph $\Gamma\subset \HH$, we define the \emph{intrinsic Dirichlet energy} of $\Gamma$ on $U$ by
\begin{equation}\label{eq:ide intro}
  E_U(\Gamma)=\frac{1}{2}\int_{U} |\nabla_f f|^2 \ud \mu,
\end{equation}
where $\mu$ is Lebesgue measure on $V_0$.  

By \cite[Thm.\ 1.6]{CMPSC}, the spherical Hausdorff measure\footnote{Unless otherwise specified, distances and Hausdorff measures in $\HH$ will be with respect to the Carnot--Carathèodory metric.} of an intrinsic Lipschitz graph $\Gamma=\Gamma_f$ of $f\from U\to \R$ can be written
\begin{equation}\label{eq:energy area coarse}
  \cS^3(\Gamma) = \int_U \sqrt{1+|\nabla_ff|^2}\ud \mu.
\end{equation}
In fact, the formula includes a multiplicative constant, but throughout this paper, we will normalize $\cS^3$ so that \eqref{eq:energy area coarse} holds without a constant.
When $\nabla_ff$ is small, 
\begin{equation}\label{eq:energy area Taylor}
  \cS^3(\Gamma)   = \int_U 1+\frac{1}{2} |\nabla_ff|^2 + O(|\nabla_ff|^3) \ud \mu = \mu(U)+E_U(\Gamma)+ O(\mu(U)\|\nabla_ff\|_\infty^3),
\end{equation}
so when $\nabla_ff$ is small, the intrinsic Dirichlet energy of $\Gamma$ is closely linked to its area.

This suggests that an area-minimizing surface with small $\nabla_ff$ should be close to energy-minimizing. We can formalize this notion using stretch automorphisms. For $r>0$, the map $s=s_{r^{-1},r}(x,y,z)=(r^{-1}x,ry,z)$ is an automorphism of $\HH$ that sends intrinsic graphs to intrinsic graphs and multiplies $\nabla_ff$ by $r^2$. Since its Jacobian on $V_0$ is $r^{-1}$, $E_{s(U)}(s(\Gamma))=r^{3}E_U(\Gamma)$. One can use this to construct energy-minimizing surfaces; if $\Gamma_1,\Gamma_2,\dots$, are area-minimizing surfaces and $\Lambda$ is such that $s_{i^{-1},i}(\Gamma_i)=\Lambda$ for all $i$, then $\Lambda$ is energy-minimizing. This shows, for instance, that if $f(x,z)=-a \sqrt{z} \sign(z)$ for some $a\ne 0$, then $\Gamma_{f}$ is energy-minimizing (Lemma~\ref{lem:herringbone}). These surfaces were studied in \cite{PaulsHMinimal, ChengHwangYang, MSCVnegative}; we call them \emph{herringbone surfaces}, since they can be written as the union of horizontal rays that branch out from the $x$--axis in a herringbone pattern.

The area-minimizing cones constructed in \cite{GR2020} also have energy-minimizing analogues. As in \cite{GR2020}, these surfaces can be written as a union of horizontal rays that either end at the origin or branch off of the singular set; the renderings in the figures are made up of these rays. The \emph{slope} of a horizontal ray is defined as the slope of its projection to the $xy$--plane; a horizontal ray is \emph{positive} or \emph{negative} depending on whether it points in the $+x$-- or $-x$--direction.
\begin{thm}\label{thm:constructEnergy}
  Let $\alpha>0$. Let $I_1,I_2,\dots\subset [-\alpha, \alpha]$ be a collection of disjoint nonempty open intervals, and let $K=[-\alpha, \alpha]\setminus \bigcup I_i$. Let $I_i=(a_i,b_i)$ and let $m_i=\frac{a_i+b_i}{2}$. Let $\Lambda_{K}\subset \HH$ be the surface consisting of the union of:
  \begin{enumerate}
  \item the negative $x$--axis, denoted $R_0$,
  \item for each $k\in K$, a positive horizontal ray from the origin with slope $k$,
  \item for each $i\ge 1$, a positive horizontal ray $R_i$ from the origin with slope $m_i$, and
  \item for each $p\in R_0$, a pair of positive horizontal rays from $p$ with slopes $\alpha$ and $-\alpha$.
  \item for each $p\in R_i, i\ge 1$, a pair of positive horizontal rays from $p$ with slopes $a_i$ and $b_i$.
  \end{enumerate}
  Then $\Lambda_{K}$ is a scale-invariant energy-minimizing surface.
\end{thm}
Since $\alpha=\sup K$ and $-\alpha=\inf K$, $\Lambda_K$ is uniquely determined by $K$.

\begin{figure}
  \begin{centering}
    \includegraphics[width=.6\textwidth]{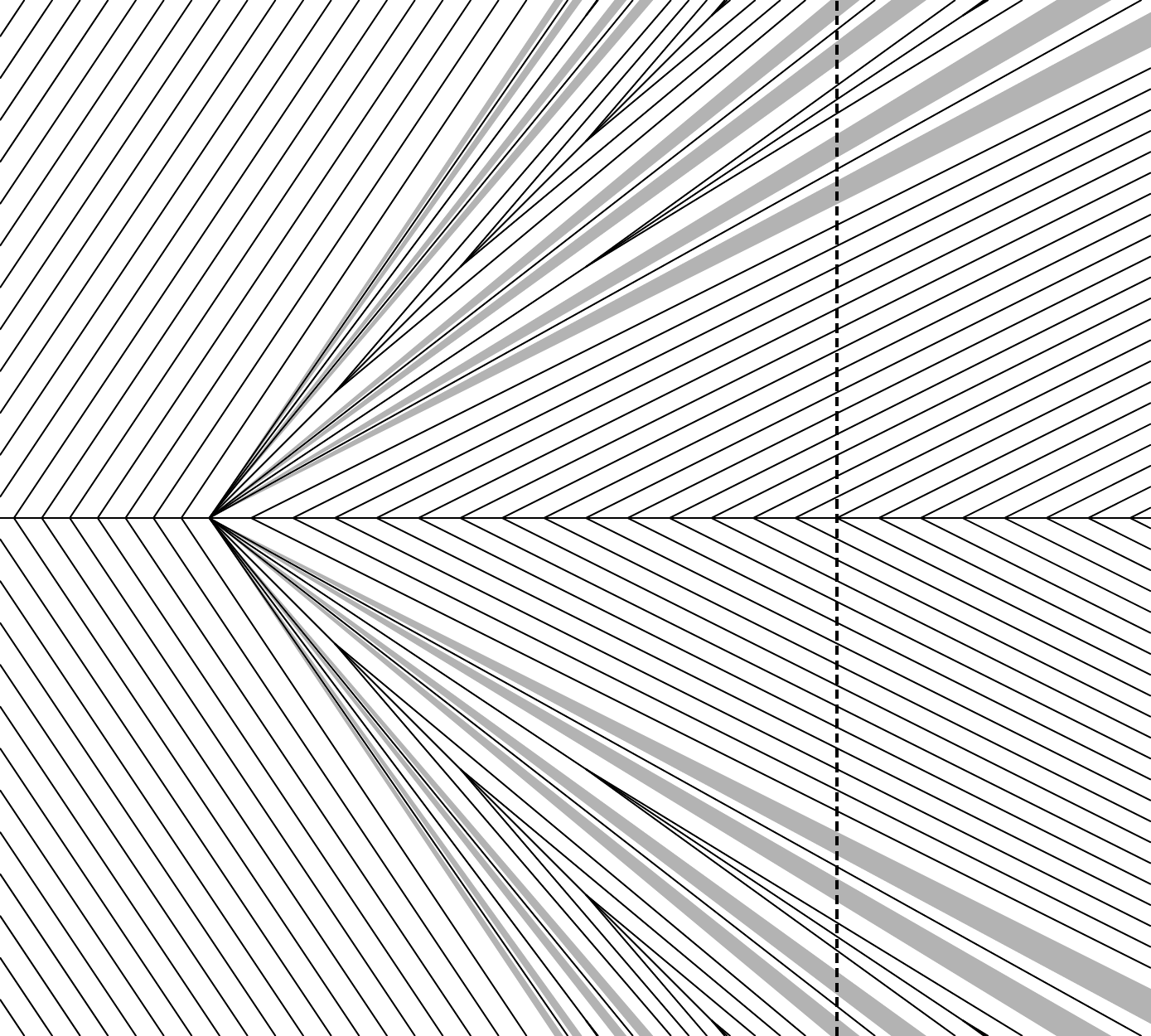}\\
    \medskip 
    \includegraphics[width=.8\textwidth]{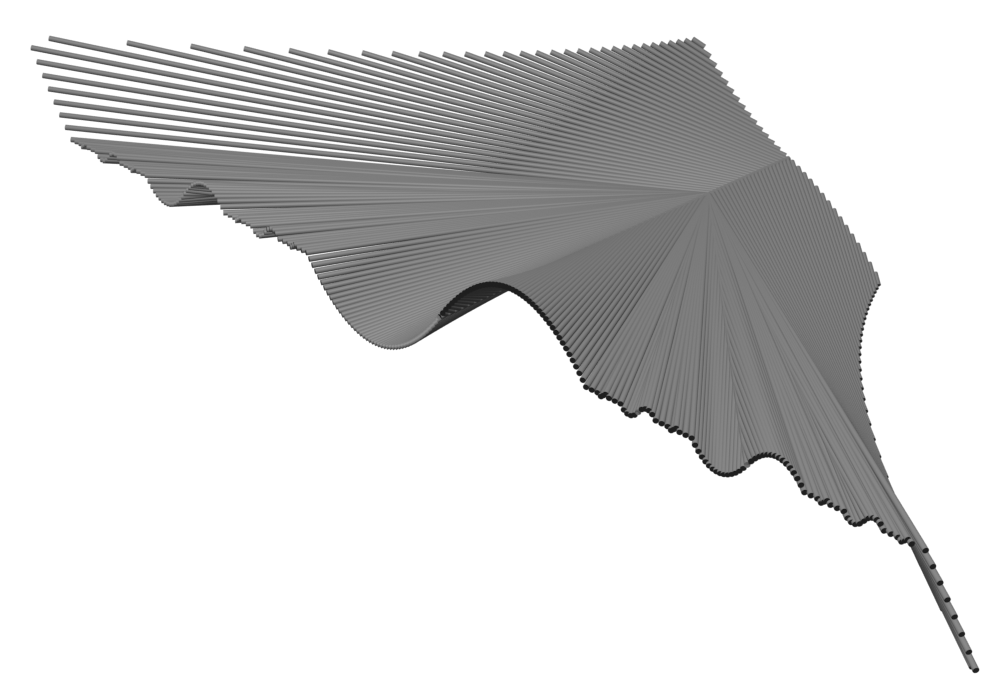}
    \medskip 
    \includegraphics[width=.7\textwidth]{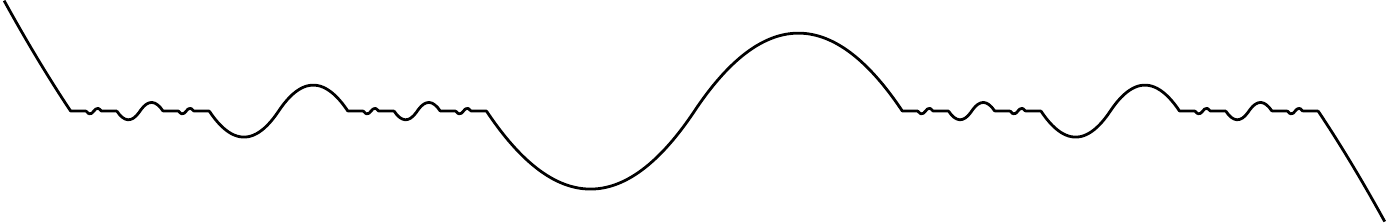}
  \end{centering}
  \caption{\label{fig:cantor}
    An energy-minimizing graph based on the middle-thirds Cantor set $K$.  The top figure shows the foliation of $\Lambda_K$ by horizontal curves, projected to the $xy$--plane. In the shaded regions, the surface coincides with the horizontal plane through $\mathbf{0}$. The bottom figure shows the cross-section of the surface in the vertical plane marked by the dashed line; the surface coincides with the $xy$--plane in the flat regions. A $3$D model of the surface can be found in the supplementary materials.
  }
\end{figure}

Any such $\Lambda_K$ can be written as a limit of stretched area-minimizing surfaces.
\begin{prop}\label{prop:stretchLimits}
  Let $\alpha$, $K=[-\alpha, \alpha]\setminus \bigcup I_i$ be as in Theorem~\ref{thm:constructEnergy}. We identify $S^1$ with the interval $(-\pi,\pi]$ so that the $+x$--axis is identified with $0$. For $n>\sqrt{\alpha}$, let
  $K_n= n^{-2} K$ and let $S_n=\Sigma_{K_n}$ be the corresponding area-minimizing surface constructed in Theorem~\ref{thm:constructArea}.

  Let $S_n^+$ and $\Lambda_K^+$ be the half-spaces bounded by $S_n$ and $\Lambda_K$. Then $\one_{S_n^+}\to \one_{\Lambda_K^+}$ in $L_1^{\mathrm{loc}}(\HH)$.
\end{prop}

Conversely, one may ask when a sequence of stretched area-minimizing surfaces has a subsequence that converges to an energy-minimizing surface. We first define the \emph{horizontal excess} of a set; the horizontal excess of the half-space bounded by a graph is related to the energy of that graph, but the horizontal excess is defined for a larger class of sets.
Let $E\subset \HH$ be a subset with finite perimeter and horizontal normal $\nu_E$. For $p\in \partial E$, $r>0$, and a horizontal vector $\nu\in S^1$, we define
\begin{align*}
  \Exc_{B(p,r)}(E,\nu)
  &=\frac{1}{r^3}\int_{B(p,r)} |\nu_E(p)-\nu|^2\ud |\partial E|(p)\\
  \Exc_{B(p,r)}(E)
  &=\inf_{\nu\in S^1}\Exc_{B(p,r)}(E,\nu),
\end{align*}
where $|\partial E|$ is the perimeter measure of $E$ (Section~\ref{sec:perim div}).
Monti \cite{MontiLipschitz} showed that if $E$ is perimeter-minimizing and if   $\Exc_{B(p,r)}(E)$ is sufficiently small for some $p\in \partial E$, $r>0$, then $\partial E$ can be approximated by an intrinsic Lipschitz graph near $p$, and the energy of this graph 

\begin{question}\label{que:aniso blowup}
  Let $c,k>0$. Suppose that $r_i\to \infty$ and $E_1,E_2,\dots\subset \HH$ are a sequence of sets that are perimeter-minimizing in the stretched balls $W_i=s_{r_i,r_i^{-1}}(B(\mathbf{0},k))$. Let $\hat{E}_i=s_{r_i,r_i^{-1}}^{-1}(E_i)$ and suppose that $\Exc_{B(\mathbf{0},k)}(\hat{E}_i,Y)<c$ for all $i$. Is it possible to choose $c$ and $k$ such that a subsequence of the $\hat{E}_i$'s converge to a set $E$ such that $\partial E$ is energy-minimizing or contact harmonic on $B(\mathbf{0},1)$?
\end{question}
(For the definition of contact harmonic, see below.) 

One might approach this question by noting that the excess $\Exc_{B}(E_i,Y)<c$ must be small for appropriate balls in $W_i$, so the results of \cite{MontiLipschitz} give an intrinsic Lipschitz half-space $\Gamma_i^+$ that approximates $E_i$. Then $\hat{\Gamma}^+_i=s_{r_i,r_i^{-1}}^{-1}(\Gamma^+_i)$ is a graph approximating $\hat{E}_i$'s, and one can hope for the corresponding functions to converge.

The problem, however, is that though the intrinsic Lipschitz constants of the $\Gamma_i$'s are uniformly bounded, the intrinsic Lipschitz constants of the $\hat{\Gamma}_i$'s are not. At best, if $h_i$ is the function such that $\hat{\Gamma}_i=\Gamma_{h_i}$, then $h_i$ satisfies an intrinsic Sobolev bound like
$$\int_{D} (\nabla_{h_i} h_i)^2 \ud \mu \le C$$
on an appropriate domain $D\subset V_0$. Not much is known about such functions, and more research may be necessary to answer Question~\ref{que:aniso blowup}.

Another approach to Question~\ref{que:aniso blowup} is to study other characterizations of energy-minimizing graphs. In an analogue to the usual notion of harmonicity, we define a \emph{contact harmonic graph} to be a critical point of the energy with respect to a class of deformations called contact deformations. In particular, energy-minimizing graphs are contact harmonic. In Section~\ref{sec:graphs and variations}, we prove some tools for working with such graphs, including first variation formulas for the energy of smooth and intrinsic Lipschitz graphs. These formulas lead to the following characterization of contact harmonicity.
\begin{thm}\label{thm:fvf intro}
  A smooth intrinsic graph $\Gamma=\Gamma_f$ is contact harmonic on $U\subset V_0$ (a critical point of the energy with respect to contact variations supported on $U$) if and only if
  \begin{equation}\label{eq:smooth harmonic}
    2 \partial_zf \cdot \nabla_f^2 f - \nabla_f^3 f = 0
  \end{equation}
  on $U$.

  For an intrinsic Lipschitz function $f$ and for $w\in C^\infty_c(U)$, let
  $$\Delta_f w = \nabla_f[\partial_x w] - \nabla_f f\cdot \partial_z w - f\nabla_f[\partial_z w].$$
  (When $f$ is smooth, $\Delta_fw=\nabla_f^2w$.) Let   
  \begin{align*}
    B_2(f, w) & = -\Delta_f w \cdot \nabla_{f}f + \frac{1}{2}  (\nabla_ff)^2\cdot \partial_z w,\\
    B_1(f, w) & = f B_2(f, w) - \frac{3}{2} (\nabla_ff)^2 \cdot \nabla_{f} w,
  \end{align*}
  Then $\Gamma=\Gamma_f$ is contact harmonic on $U$ if and only if
  $$\int_U B_1(f, w)\ud\mu = \int_U B_2(f, w)\ud \mu = 0$$
  for every $w\in C^\infty_c(U)$. 
\end{thm}
This theorem follows from Theorem~\ref{thm:fvf domain smooth} (in the smooth case) and Theorem~\ref{thm:fvf domain lip} (in the intrinsic Lipschitz case).

Condition \eqref{eq:smooth harmonic} is weaker than the condition that the horizontal mean curvature of $\Gamma_f$ vanishes (i.e., $\nabla_f^2f=0$). Indeed, Theorem~\ref{thm:fvf intro} implies that the surface $\Gamma_f=\{(x,y,z)\in \HH\mid y=x^2\}$, which has $\nabla_f^2f=2$, is contact harmonic. This happens because contact harmonic graphs are energy-stationary only for contact diffeomorphisms, not arbitrary diffeomorphisms.  Contact diffeomorphisms preserve horizontality, so they also preserve horizontal connectivity. That is, if $\phi\from \HH\to \HH$ is a contact diffeomorphism and if $p,q\in \Gamma$ are connected by a horizontal curve in $\Gamma$, then $\phi(p)$ and $\phi(q)$ are connected by a horizontal curve in $\phi(\Gamma)$. A contact harmonic graph may admit smooth deformations that reduce its energy, but these smooth deformations affect the horizontal connectivity of the graph. \footnote{Manuel Ritoré informs us that the cylinder $\{x^2+y^2=r^2\}$ is an area-stationary analogue of this example, i.e. a smooth surface in $\HH$ with nonzero horizontal mean curvature that is area-stationary under contact deformations, but as far as we are aware, the calculation has not been published.}

We use the first-variation formula for smooth graphs to prove a first variation formula for piecewise smooth graphs (Theorem~\ref{thm:fvf herringbone}). This shows that if $\Gamma$ is a piecewise smooth contact harmonic graph whose characteristic nexus is a curve and if the foliation by horizontal curves is well-behaved near the characteristic nexus, then the slope of the characteristic nexus is the average of the slopes of the horizontal curves on either side of the nexus. Such singularities are similar to the singularities of $H$--minimal surfaces studied in \cite{PaulsHMinimal, ChengHwangYang}. 

Finally, we conclude with an observation. One motivation for this work was the Bernstein problem for surfaces with low regularity. Bernstein's Theorem states that an area-minimizing codimension--$1$ graph in $\R^n$ is a plane. The Bernstein problem asks for conditions under which other classes of area-minimizing surfaces are planes. Barone Adesi, Serra Cassano, and Vittone showed that if $\Gamma\subset \HH$ is a area-minimizing intrinsic graph of an entire $C^2$ function, then $\Gamma$ is a vertical plane \cite{BASCV}. The same result for $C^1$ functions was proved by Galli and Ritoré \cite{GRBernstein}, and for locally Lipschitz functions by Nicolussi Golo and Serra Cassano \cite{NSCBernstein}, but there are many examples of area-minimizing intrinsic Lipschitz graphs that are not vertical planes, such as many of the examples in Section~\ref{sec:calibrations}. All of these examples, however, are also entire $Z$--graphs, i.e., sets of the form $\{(x,y,g(x,y))\in \HH\mid x,y \in \R\}$. One may then ask the following:
\begin{question}\label{que:Bernstein}
  Is every nonplanar entire area-minimizing (or energy-minimizing) intrinsic Lipschitz graph also a $Z$--graph?
\end{question}

\begin{remark}
  The 3D models used in the figures in this paper can be found in .obj format in the ancillary files on the arXiv page for this paper. These files can be opened in Preview on Macs, Paint 3D on Windows, or any 3D modeling program. They can also be found at \url{https://cims.nyu.edu/~ryoung/harmonic/}.
\end{remark}

\subsection{Outline of paper}
In Section~\ref{sec:prelims}, we establish notation for the Heisenberg group and prove some integration formulas which will be used throughout the paper. In Section~\ref{sec:calibrations}, we define the intrinsic Dirichlet energy, prove some basic properties of energy-minimizing surfaces, and describe a method to prove that a surface is energy-minimizing using a calibration. We use this method to prove Theorem~\ref{thm:constructEnergy} and Proposition~\ref{prop:stretchLimits}. In Section~\ref{sec:graphs and variations}, we define contact variations and contact harmonicity and prove first variation formulas for the energy of smooth graphs and intrinsic Lipschitz graphs, including Theorem~\ref{thm:fvf intro}. Many energy-minimizing and area-minimizing surfaces have singularities along curves, which we call \emph{herringbone singularities}, and we prove a first variation formula for the energy of such a graph in Section~\ref{sec:fvf herring}. 

The author would like to thank Roberto Monti, Sebastiano Nicolussi Golo, Manuel Ritoré, and Davide Vittone for their helpful discussions during the preparation of this paper and to thank the Institute of Advanced Study for its hospitality.
\section{Preliminaries}\label{sec:prelims}
Throughout this paper, we use the following standard conventions for asymptotic notation. The notations $f \lesssim g$ and $g \gtrsim f$ mean that $f\le C g$ for some universal constant $C\in (0,\infty)$, and $f\approx g$ means that $f \lesssim g$ and $g \lesssim f$.  If $C$ depends on some parameters, we indicate this by subscripts, i.e., $f\lesssim_{t}g$ implies that $f\le C g$ where $C=C(t)$ depends only on $t$. Finally, we use big--$O$ notation $O(f)$ to denote an error term whose absolute value is bounded by a universal constant multiple of $|f|$.  

\subsection{The Heisenberg group}
The Heisenberg group $\HH$ is the $3$--dimensional simply connected Lie group with Lie algebra
$$\mathfrak{h}:=\langle X, Y, Z\mid [X,Y]=Z, [X,Z]=[Y,Z]=0\rangle.$$
We identify $\HH$ with its Lie algebra via the Baker--Campbell--Hausdorff formula, i.e., $\HH=\langle X,Y,Z\rangle \cong \R^3$ and 
$$(x,y,z)\cdot(x',y',z')=\left(x+x',y+y',z+z'+\frac{xy'-yx'}{2}\right).$$
We use $X$, $Y$, and $Z$ to denote the coordinate vectors of $\R^3$ and the corresponding left-invariant fields $X_{(x,y,z)}=(1,0,-\frac{y}{2})$, $Y_{(x,y,z)}=(0,1,\frac{x}{2})$, and $Z_{(x,y,z)}=(0,0,1)$. We write the coordinate vector fields in $\R^3$ as $\partial_x$, $\partial_y$, and $\partial_z=Z$.  Every vector $v\in \HH$ generates a one-parameter subgroup of $\HH$; we write $\langle v\rangle=\R v$ for this subgroup and define $v^t=tv$ for all $t\in \R$.

Let $\mathsf{A}=\spanop(X,Y)$ be the left-invariant distribution spanned by $X$ and $Y$; we refer to this as the \emph{horizontal distribution}. 
Curves that are Lipschitz (with respect to the Euclidean metric) and almost everywhere tangent to $\mathsf{A}$ are called \emph{horizontal curves}. For $h\in \HH$ and $(a,b)\in \R^2\setminus (0,0)$, we call the line $L=h\cdot \langle aX+bY\rangle$ a \emph{horizontal line}. Let $\pi\from \HH\to \R^2$ be the projection $\pi(x,y,z)=(x,y)$; we will identify $\R^2$ with the $xy$--plane in $\HH$. The \emph{slope} of a horizontal line $L$ is the slope of the projection $\pi(L)$, i.e., $\frac{b}{a}$, as long as $a\ne 0$. A \emph{vertical plane} $V\subset \HH$ is a plane parallel to $\langle Z\rangle$; its slope is the slope of the projection $\pi(V)$.

Let $d$ represent the \emph{Carnot--Carathéodory metric} on $\HH$. That is, we define a norm on $\mathsf{A}$ by $\|aX+bY\|=\sqrt{a^2+b^2}$, and given a horizontal curve $\gamma\from I\to \HH$, we define the length of $\gamma$ as
$$\ell(\gamma)=\ell_{\R^2}(\pi\circ \gamma)=\int_I \|\gamma'(t)\|\ud t.$$
For $p,q\in \HH$, let $d(p,q)=\inf_{\gamma}\ell(\gamma)$, where the infimum is taken over all horizontal curves $\gamma\from [0,1]\to \HH$ such that $\gamma(0)=p$ and $\gamma(1)=q$. This is a left-invariant metric that satisfies the ball-box inequality
\begin{equation}\label{eq:ballbox}
  d(\mathbf{0},p)\approx \max\{|x(p)|, |y(p)|, \sqrt{|z(p)|}\}.
\end{equation}
This metric gives $\HH$ Hausdorff dimension $4$, and a smooth codimension--$1$ surface or intrinsic Lipschitz graph (Definition~\ref{def:intrinsic lipschitz graph}) has Hausdorff dimension $3$. Let $\cS^d$ be the spherical Hausdorff $d$--measure, normalized so that $\cS^4$ is Lebesgue measure on $\HH$ and the restriction of $\cS^3$ to the $xz$--plane is Lebesgue measure on the $xz$--plane.

\subsection{Intrinsic graphs}
Let $V_0=\{(x,y,z)\in \HH\mid y=0\}$ be the $xz$--plane.  For any $U\subset V_0$ and any $f\from U \to \R$, the \emph{intrinsic graph} of $f$ is the set
$$\Gamma_f=\{v\cdot Y^{f(v)}\mid v\in U\}.$$
Let $\Psi_f\from U\to \Gamma_f$ be the map $\Psi_f(v)=v\cdot Y^{f(v)}$.  Let $\Pi\from \HH\to V_0$,
\begin{equation}\label{eq:def Pi}
  \Pi(x,y,z) = (x,y,z)\cdot Y^{-y} = \left(x, 0, z - \frac{xy}{2}\right).
\end{equation}
This is the projection to $V_0$ whose fibers are the cosets of $\langle Y\rangle$. The restrictions $\Pi|_{\Gamma_f}$ and $\Psi_f|_{U}$ are bijections (homeomorphisms when $f$ is continuous), and $\Pi \circ \Psi_f=\id_{U}$.

Let $\gamma=(\gamma_x,\gamma_y,\gamma_z)\from I\to \HH$ be a horizontal curve such that $\gamma$ has \emph{unit $x$--speed}, i.e., $\gamma_x'=1$. For almost every $t\in I$,
$$\gamma'(t)=\gamma'_x(t) X_{\gamma(t)} + \gamma'_y(t) Y_{\gamma(t)} = X_{\gamma(t)} + \gamma'_y(t) Y_{\gamma(t)}.$$
For any $p\in \HH$, we have $\Pi_*(X_{p}) =\partial_x - y(p) \partial_z$ and $\Pi_*(Y_{p}) =0$, so
\begin{equation}\label{eq:horiz proj}
  (\Pi\circ \gamma)'(t) = \Pi_*(X_{\gamma(t)}) + \gamma'_y(t) \Pi_*(Y_{\gamma(t)}) = \partial_x - \gamma_y(t) \partial_z.
\end{equation}
That is, one can reconstruct $\gamma$ from the projection $\Pi\circ \gamma$.

Franchi, Serapioni, and Serra Cassano \cite{FSSC06} defined the class of intrinsic Lipschitz graphs, which are analogues of graphs of Lipschitz functions in $\R^n$.  We will give a definition which is equivalent to theirs, but with a different value for the Lipschitz constant \cite{Rigot}.  
\begin{defn}\label{def:intrinsic lipschitz graph}
  Let $\Gamma=\Gamma_f\subset \HH$ be an intrinsic graph and let $\lambda \in (0,1)$.  We say that $\Gamma$ is an \emph{intrinsic $\lambda$--Lipschitz graph} (and that $f$ is an \emph{intrinsic $\lambda$--Lipschitz function}) if for every $p,q\in \Gamma$,
  $$|y(q)-y(p)| \le \lambda d(p,q).$$
\end{defn}
If $\Pi(\Gamma_f)=V_0$, or equivalently, if the domain of $f$ is all of $V_0$, we say that $\Gamma_f$ is an \emph{entire} intrinsic Lipschitz graph. Every intrinsic Lipschitz graph is a subset of an entire intrinsic Lipschitz graph; see \cite[Thm.\ 27]{NY18} and \cite{Rigot}.

Let $\Gamma=\Gamma_f$ be an intrinsic Lipschitz graph. Let $\nabla_f$ be the vector field on $V_0$ given by $\nabla_f=\partial_x - f \partial_z$. We call the corresponding operator the \emph{intrinsic gradient}.  For every point $p\in \Gamma$, there is a (possibly non-unique) unit $x$--speed horizontal curve $\gamma=(\gamma_x,\gamma_y,\gamma_z)\from \R\to \Gamma$ such that $\gamma(0)=p$. We call the projection $\lambda=\Pi\circ \gamma$ of such a curve to $V_0$ a \emph{characteristic curve} of $\Gamma$. By \eqref{eq:horiz proj}, every characteristic curve is an integral curve of $\gamma$; conversely, by Theorems~1.1 and 1.2 of \cite{BiCaSC}, every integral curve of $\gamma$ is the projection of a unit $x$--speed horizontal curve. Thus the characteristic curves are exactly the integral curves of $\nabla_f$. When $f$ is smooth, there is a unique characteristic curve passing through every point of $V_0$, and the intrinsic gradient is the derivative along these characteristic curves.

The intrinsic gradient $\nabla_ff$ of $f$ is particularly important. When $f$ is smooth and $v\in V_0$,  $\nabla_ff(v)$ is the slope of the unique horizontal curve through $p=\Psi_f(v)$. If $s_{t,t}(x,y,z)=(tx,ty,t^2z)$ is the scaling automorphism, then $s_{t,t}(p^{-1}\Gamma)$ converges to a vertical plane with slope $\nabla_ff(v)$ as $t\to \infty$. When $f$ is intrinsic Lipschitz, the horizontal curve through $p$ may not be unique, and the derivative $\nabla_ff$ may not be defined everywhere. We thus define $\nabla_f f$ distributionally as follows.
\begin{defn}\label{def:distrib gradient}
  If $f\from V_0\to \R$ is continuous, we say that $\nabla_ff$ exists in the sense of distributions if there is a function $\theta \in L_\infty^{\mathrm{loc}}$ such that for every $\psi\in C^1_c$,
  $$\int_{V_0} \theta \psi \ud \mu= \int_{V_0} -f  \partial_x\psi + \frac{f^2}{2}\partial_z\psi\ud \mu.$$
  If so, we write $\nabla_ff=\theta$.  When $f$ is $C^1$, this coincides with the previous definition.
\end{defn}
The intrinsic gradient of an intrinsic $\lambda$--Lipschitz function is bounded by a function of $\lambda$. Conversely, if $U\subset V_0$ is an open set and $f\from U\to \R$ is $C^1$ and satisfies $\|\nabla_ff\|_{L_\infty}<L$, then $f$ is locally intrinsic Lipschitz with a constant depending on $L$ \cite[Prop.\ 1.8]{CMPSC}.

Citti, Manfredini, Pinamonti, and Serra Cassano showed that intrinsic Lipschitz graphs can be approximated by smooth graphs.
\begin{thm}[{\cite[Thm.\ 1.7]{CMPSC}}]\label{thm:CMPSC approx}
  Let $f\from V_0\to \R$ be an intrinsic Lipschitz function and let $\omega \subset V_0$ be a bounded open set. Then there exists a sequence of functions $f_k\in C^\infty(\omega)$ such that
  \begin{enumerate}
  \item $f_k\to f$ uniformly in $\omega$,
  \item $|\nabla_{f_k}f_k(v)|\le \|\nabla_ff\|_{L_\infty(\omega)}$ for all $v\in \omega$, and
  \item $\nabla_{f_k}f_k(v)\to \nabla_ff(v)$ for $\mu$--a.e.\ $v\in \omega$.
  \end{enumerate}
\end{thm}

Finally, the following integration rules for $\nabla_f$ will be helpful.  
\begin{lemma}\label{lem:nabla f path rule}
  Let $U\subset V_0$ be a closed bounded set with piecewise smooth boundary and let $f,g \from U \to \R$ be functions which are smooth on the interior of $U$ and continuous on $\partial U$.  Suppose that $\partial_z f, \partial_z g, \nabla_fg\in L_1(U)$.
  Let $\alpha$ parametrize $\partial U$ in the positive direction.  Then
  $$\int_U \nabla_f g \ud \mu = \int_{\partial U} (fg, g)\cdot \ud \alpha + \int_U g \cdot \pd{f}{z}\ud \mu.$$
  If $\partial U$ consists of segments on which $g$ vanishes and characteristic curves of $\Gamma_f$, then the first term vanishes and 
  $$\int_U \nabla_f g \ud \mu = \int_U g \cdot \pd{f}{z}\ud \mu.$$ 
\end{lemma}
\begin{proof}
  On any subset $\omega\subset U$ with piecewise smooth boundary $\beta$, Green's Theorem and integration by parts imply that
  \begin{align*}
    \int_\omega \nabla_f g\ud \mu 
    &= \int_\omega \pd{g}{x} \ud \mu - \int_\omega f \cdot \pd{g}{z}\ud \mu \\
    &= \int_\omega \pd{g}{x} - \partial_z[fg] \ud \mu + \int_\omega g \cdot \pd{f}{z}\ud \mu \\
    &= \int_{\partial \omega} (f g, g)\cdot \ud \beta + \int_U g \cdot \pd{f}{z}\ud \mu.
  \end{align*}
  We have $\partial_z[fg]\in L_1(U)$ and $\partial_xg = \nabla_fg+f\partial_zg\in L_1(U)$, so the lemma follows by letting $\omega\to U$. 
\end{proof}

As a corollary, we get an analogue of the integration by parts formula.
\begin{remark}
  In the corollary below and throughout this paper, we use square brackets to indicate the function acted on by a differential operator.  We use $\cdot$ as a low-precedence multiplication operator, so that $\nabla_fg\cdot h$ represents $(\nabla_fg) h$, not $\nabla_f[gh]$.
\end{remark}

\begin{cor}\label{cor:integration by parts}
  Let $U, f$, and $\alpha$ be as above and let $g, h\from U\to \R$ be smooth functions which are smooth on the interior of \(U\) and continuous on \(\partial U\), and satisfy $\partial_zf, \partial_z[gh], \nabla_fg, \nabla_fh \in L_1(U)$.  Then
  \begin{align*}
    \int_{U} \nabla_f g\cdot h \ud \mu 
    &= \int_{U} \nabla_f[gh] - \nabla_fh\cdot g \ud \mu \\ 
    &= \int_{\partial U} (f g h, g h)\cdot \ud \alpha + \int_U gh\cdot \partial_zf\ud\mu - \int_U \nabla_f h \cdot g \ud \mu.
  \end{align*}
  If \(\partial U\) consists of segments on which $gh$ vanishes and characteristic curves of \(\Gamma_f\), then
  $$\int_{U} \nabla_fg\cdot  h \ud \mu = \int_U gh\cdot \partial_zf  - g\cdot \nabla_f h \ud \mu.$$
\end{cor}

\subsection{Automorphisms}

The Heisenberg group admits families of stretch and shear automorphisms.  For $a,b\in \R\setminus \{0\}$, the \emph{stretch automorphism} $s_{a,b}\from \HH\to \HH$ is given by
$$s_{a,b}(x,y,z)=(ax,by,abz).$$
When $a=b$, this acts as a scaling on $d$, i.e., $d(s_{t,t}(p),s_{t,t}(q))=td(p,q)$. The \emph{shear automorphism} $P_b\from \HH\to \HH$ is given by 
$$P_b(x,y,z)=(x,y+bx,z).$$
Both of these are group automorphisms of $\HH$ that send the horizontal distribution to itself and thus send horizontal curves to horizontal curves. In fact, they send cosets of $\langle Y\rangle$ to cosets of $\langle Y\rangle$, so the image of an intrinsic graph is an intrinsic graph. The following lemma holds.

\begin{lemma}\label{lem:graph autos}
  Let $\Gamma=\Gamma_f$ be the intrinsic graph of a function $f\from U\subset V_0\to \R$. Let $q\from \HH\to \HH$ be a stretch map, shear map, or left translation and let $\hat{q}\from V_0\to V_0$, $\hat{q}(v)=\Pi(q(v))$ be the map that $q$ induces on $V_0$.  Then there is a function $\hat{f}\from \hat{q}(U)\to \R$ such that $\Gamma_{\hat{f}}=q(\Gamma)$. 
  \begin{itemize}
  \item
    If $a, b\in \R\setminus \{0\}$ and $q=s_{a,b}$, then for all $v\in U$ , $\hat{f}(\hat{q}(v))=b f(v)$ and
    $$\nabla_{\hat{f}}\hat{f}(\hat{q}(v))=\frac{b}{a}\nabla_ff(v).$$
  \item
    If $b\in \R$ and $q=P_b$, then for all $v\in U$, $\hat{f}(\hat{q}(v))=f(v)+bx(v)$ and 
    $\nabla_{\hat{f}}\hat{f}(\hat{q}(v))=\nabla_ff(v)+b$.
  \item
    If $h\in \HH$ and $q(p)=hp$, then for all $v\in U$, $\hat{f}(\hat{q}(v))=f(v)+y(h)$ and $\nabla_{\hat{f}}\hat{f}(\hat{q}(v))=\nabla_ff(v)$.
  \end{itemize}
\end{lemma}
\begin{proof}
  In all three cases, one can calculate that $q(vY^{f(v)})=\hat{q}(v)Y^{\hat{f}(v)}$, so $\Gamma_{\hat{f}}=q(\Gamma)$.

  By Theorem~\ref{thm:CMPSC approx}, it suffices to prove the formulas for $\nabla_{\hat{f}}\hat{f}$ when $f$ is smooth. Let $v\in V_0$ and let $L$ be the horizontal line tangent to $\Gamma$ at $\Psi_f(v)$; this line has slope $S=\nabla_ff(v)$. The image $q(L)$ is tangent to $q(\Gamma)$ at $q(\Psi_f(v))=\Psi_{\hat{f}}(\hat{q}(v))$, so $\nabla_{\hat{f}}\hat{f}(\hat{q}(v))$ is equal to the slope of $q(L)$. If $q=s_{a,b}$, then $q(L)$ has slope $\frac{b}{a}S$, if $q=P_b$, then $q(L)$ has slope $S+b$, and if $q(p)=hp$, then $q(L)$ has slope $S$, as desired.
\end{proof}

\subsection{Perimeter and divergence}\label{sec:perim div}
Let $C^1(\HH;\mathsf{A})$ be the set of $C^1$ horizontal vector fields. For $V\in C^1(\HH,\mathsf{A})$, $V=v_1X+v_2Y$, we define the \emph{horizontal divergence} of $V$ as $\div_{\HH} V=Xv_1+Yv_2$. For a measurable subset $E\subset \HH$, we say that $E$ has \emph{locally finite perimeter} if for any bounded open set $U\subset \HH$, 
$$|\partial E|(U) \eqdef \sup\left\{\int_E \div_\HH V\ud \cS^4\mid V\in C^1_c(U;\mathsf{A}), |V|\le 1\right\}<\infty.$$
We call $|\partial E|(U)$ the \emph{perimeter} of $E$ in $U$. Franchi, Serapioni, and Serra Cassano \cite{FSSCRectifiability} showed that if $E$ has locally finite perimeter, then $|\partial E|$ is a Radon measure, and there is a $|\partial E|$--measurable horizontal vector field $\nu_E$ (the \emph{horizontal inward unit normal} to $\partial E$) such that $|\nu_E(p)|=1$ for $|\partial E|$--a.e.\ $p$ and such that for every $V\in C^1_c(\HH;\mathsf{A})$,
\begin{equation}\label{eq:div thm}
  \int_E \div_\HH V\ud\cS^4 =\int_\HH \langle \nu_E,V\rangle \ud |\partial E|.
\end{equation}

The measure $|\partial E|$ is concentrated on a subset $\partial^*E\subset \partial E$ called the \emph{reduced boundary}. When $E$ has locally finite perimeter, $\partial^*E$ is contained in the \emph{measure-theoretic boundary} $\partial_{*,\cS^4}E$, where $p\in \partial_{*,\cS^4}E$ if and only if
$$\limsup_{r\to 0} \frac{\cS^4(E \cap B(p,r))}{\cS^4(B(p,r))}>0 \text{\qquad and \qquad} \limsup_{r\to 0} \frac{\cS^4(B(p,r)\setminus E)}{\cS^4(B(p,r))}>0.$$
By \cite{FSSCRectifiability},
\begin{equation}\label{eq:concentrate perimeter}
  |\partial E|= \cS^3\rstr \partial^*E = \cS^3\rstr \partial_{*,\cS^4}E.
\end{equation}

The perimeter measure can also be characterized in terms of BV functions. By \cite{GaroNhieuIsopSob}, $E$ has locally finite perimeter if and only if $\one_E$ is a locally BV function. In this case, the distributional gradient satisfies
\begin{equation}\label{eq:dist grad}
  \nabla \one_E=\nu_E |\partial E|.
\end{equation}
If $E,F\subset \HH$ have locally finite perimeter, then $E\cap F$ has locally finite perimeter, and $\nu_{E\cap F}$ agrees with $\nu_E$ or $\nu_F$ everywhere that both normals are defined.
\begin{lemma}\label{lem:intersection perimeter}
  Let $E, F\subset \HH$ have locally finite perimeter. Then $E\cap F$ has locally finite perimeter. There is a $\cS^3$--null set $K$ such that
  \begin{equation}\label{eq:rb inter}
    \partial^*(E\cap F) \setminus K \subset \partial^*E \cup \partial^* F.
  \end{equation}
  Then  $\nu_{E\cap F}(p)=\nu_F(p)$ for all $p\in \partial^*(E\cap F) \cap \partial^*F$ and
   $\nu_{E\cap F}(p)=\nu_E(p)$ for all $p\in \partial^*(E\cap F) \cap \partial^*E$.
\end{lemma}
\begin{proof}
  Since $\one_E+\one_F\in \BV(\HH)$, the coarea formula (Theorem~5.2 of \cite{GaroNhieuIsopSob}), implies that $E\cap F=\{p\in \HH\mid \one_E(p)+\one_F(p)>1\}$ has locally finite perimeter. We have
  \begin{equation}\label{eq:mtb inter}
    \partial_{*,\cS^4}(E\cap F)\subset \partial_{*,\cS^4}E\cap \partial_{*,\cS^4}F,
  \end{equation}
  and the measure-theoretic and reduced boundaries of $E$, $F$, and $E\cap F$ agree up to a null set, so \eqref{eq:mtb inter} implies \eqref{eq:rb inter}.

  Let $p\in \partial^*(E\cap F) \cap \partial^*F$. For any $r>0$, let $(E\cap F)_r = s_{r^{-1},r^{-1}}\left(p^{-1}(E\cap F)\right)$, and let $F_r = s_{r^{-1},r^{-1}}(p^{-1}F)$. For $V\in \mathsf{A}$, let $P^+_V$ be the half-space
  $$P^+_V=\{h\in \HH \mid \langle V,\pi(h)\rangle>0\}.$$
  By Theorem~4.1 of \cite{FSSCRectifiability}, 
  $$\lim_{r\to 0} \one_{F_r} = \one_{P_{\nu_F(p)}^+} \text{\qquad and\qquad} \lim_{r\to 0} \one_{(E\cap F)_r} = \one_{P_{\nu_{E\cap F}(p)}^+}$$
  in $L_{1}^{\mathrm{loc}}(\HH)$.  Since $\lim_{r\to 0} \one_{(E\cap F)_r} \le \lim_{r\to 0} \one_{F_r}$, this implies $\nu_{E\cap F}(p)=\nu_{F}(p)$. Swapping $E$ and $F$, we find that $\nu_{E\cap F}(p)=\nu_E(p)$ for all $p\in \partial^*(E\cap F) \cap \partial^*E$.
\end{proof}

\section{Energy-minimizing surfaces and calibrations}\label{sec:calibrations} 
Let $f\from V_0\to \R$ be an intrinsic Lipschitz function.  Let $U\subset V_0$ be a bounded measurable subset.  We define the \emph{intrinsic Dirichlet energy} of \(f\) on \(U\) by
$$E_U(f)=\frac{1}{2}\int_{U} (\nabla_f f)^2 \ud \mu.$$
For $\Gamma=\Gamma_f$ and for any bounded open subset $W\subset \HH$, we define
$E_W(\Gamma)=E_{\Pi(W\cap \Gamma)}(f).$

In this section, we use calibrations to construct surfaces that minimize intrinsic energy subject to some boundary conditions. We first give a definition and some basic properties.

\begin{defn}
  Let $f\from V_0\to \R$ be an intrinsic Lipschitz function and let $U\subset V_0$ be an open subset.
  We say that $f$ is an \emph{energy-minimizing function} on $U$ if
  $$E_{U\cap B(p,r)}(f)\le E_{U\cap B(p,r)}(g)$$
  for every $p\in V_0$, $r>0$, and every intrinsic Lipschitz function \(g\) such that $f-g\in C^c_0(U\cap B(p,r))$.
  
  For an entire intrinsic Lipschitz graph $\Gamma=\Gamma_f$, we define the \emph{epigraph} of $\Gamma$ to be
  $$\Gamma^+=\{p Y^t\in \HH\mid p\in \Gamma, t\ge 0\}.$$
  We say that $\Gamma$ is an \emph{energy-minimizing graph} on an open subset $W\subset \HH$ if
  $$E_{W\cap B(p,r)}(\Gamma)\le E_{W\cap B(p,r)}(\Lambda)$$
  for every $p\in \HH$, $r>0$, and every entire intrinsic Lipschitz graph $\Lambda$ such that $\Gamma^+\symdiff \Lambda^+\Subset W\cap B(p,r)$.  Here, $S\symdiff T$ is the symmetric difference $S\symdiff T=(S\setminus T)\cup (T\setminus S)$.    
  In particular, $f$ is energy-minimizing on $U$ if and only if $\Gamma_f$ is energy-minimizing on $\Pi^{-1}(U)$. 
\end{defn}

While perimeter minimization is preserved by scalings and rotations, energy minimization is preserved by stretches and shears.
\begin{lemma}\label{lem:least-energy symmetry}
  Let $\Gamma=\Gamma_f\subset \HH$ be a graph that minimizes energy on an open set $W\subset \HH$ and let $h\from \HH \to \HH$ be a left-translation, stretch, or shear.  Then $h(\Gamma)$ minimizes energy on $h(W)$.
\end{lemma}

We need the following calculation.
\begin{lemma}\label{lem:average naf}
  Let $U\subset V_0$ be a bounded open set and let $f,g\from V_0\to \R$ be intrinsic Lipschitz functions such that $f-g\in C^0_c(U)$. Then
  $$\int_U \nabla_f f \ud \mu = \int_U \nabla_g g\ud \mu.$$
\end{lemma}
\begin{proof}
  Let $S=\supp(f-g)\Subset U$ and let $\psi\in C^\infty(U)$ be a function such that $\psi(s)=1$ for all $s\in S$ and $\psi(u)=0$ for all $u\in U\setminus S$.  By Definition~\ref{def:distrib gradient},
  \begin{align*}
    \int_U \nabla_f f\ud \mu 
    &= \int_U \nabla_f f \cdot \psi \ud \mu + \int_U \nabla_f f \cdot (1-\psi) \ud \mu \\ 
    &= \int_{U} -f  \partial_x\psi + \frac{f^2}{2}\partial_z\psi\ud \mu + \int_U \nabla_f f \cdot (1-\psi) \ud \mu.
  \end{align*}
  Since $\psi$ is constant on $S$ and $\psi=1$ on $S$, both integrals are zero on $S$. Therefore, since $f=g$ on $U\setminus S$,
  $$\int_U \nabla_f f\ud \mu 
  = \int_{U\setminus S} -f  \partial_x\psi + \frac{f^2}{2}\partial_z\psi\ud \mu + \int_{U\setminus S} \nabla_f f \cdot (1-\psi) \ud \mu=\int_U \nabla_g g\ud \mu,$$
  as desired.
\end{proof}

Now we prove the lemma.
\begin{proof}[{Proof of Lemma~\ref{lem:least-energy symmetry}}]
  It suffices to show that for every bounded open set $U\subset W$ and every intrinsic Lipschitz graph  $\Lambda$ such that $\Lambda^+\symdiff h(\Gamma^+)\Subset h(U)$, we have $E_{h(U)}(h(\Gamma))\le E_{h(U)}(\Lambda)$.

  When $h$ is a translation and $\Lambda^+\symdiff h(\Gamma^+)\Subset h(U)$, the minimality of $\Gamma$ implies
  $$E_{h(U)}(h(\Gamma))=E_U(\Gamma)\le E_U(h^{-1}(\Lambda)) =E_{h(U)}(\Lambda),$$
  as desired.
  
  Suppose that $h=s_{a,b}$ for some $a,b\ne 0$. By Lemma~\ref{lem:graph autos}, for any intrinsic Lipschitz graph $\Sigma=\Gamma_\phi$, we have
  \begin{equation}\label{eq:stretch energy}
    E_{h(U)}(h(\Sigma)) = \frac{b^2}{a^2} \frac{\mu(h(U))}{\mu(U)} E_U(\Sigma) = \frac{b^2}{a}E_U(\Sigma).
  \end{equation}
  If $\Lambda^+\symdiff h(\Gamma^+)\Subset h(U)$, the minimality of $\Gamma$ implies
  $$E_{h(U)}(h(\Gamma))=\frac{b^2}{a}E_U(\Gamma)\le \frac{b^2}{a} E_U(h^{-1}(\Lambda)) =E_{h(U)}(\Lambda).$$

  Now suppose that $h$ is a shear, i.e., $h=P_b$ for some $b\in \R$. Let $\Lambda=\Gamma_{\hat{g}}$ be an intrinsic Lipschitz graph such that $\Lambda^+\symdiff h(\Gamma^+)\Subset h(W)$. Let $g$ be the intrinsic Lipschitz function such that $h^{-1}(\Lambda)=\Gamma_g$. Then $\Gamma_g^+\symdiff \Gamma^+\Subset W$, and for any $v\in V_0$, either $f(v)=g(v)$ (in which case $\Psi_f(v)=\Psi_g(v)$) or the horizontal segment from $\Psi_f(v)$ to $\Psi_g(v)$ is contained in $W$. In either case, $\Psi_f(v)\in W$ if and only if $\Psi_{g}(v)\in W$, i.e., $\Pi(W\cap \Gamma)=\Pi(W\cap \Gamma_g)$. Let $D_0=\Pi(W\cap \Gamma)$ and let $D$ be a bounded open subset such that $D_0\Subset D$ and thus $f-g\in C^0_c(D)$. The minimality of $f$ implies that $E_D(f)\le E_D(g)$.

  Let $\hat{h}\from V_0\to V_0$,
  $$\hat{h}(x,0,z)=\Pi(h(x,0,z))=\big(x,0,z-\frac12 bx^2\big).$$
  This is measure-preserving on $V_0$, and by Lemma~\ref{lem:graph autos}, for any   intrinsic Lipschitz graph $\Sigma=\Gamma_\phi$, there is a $\hat{\phi}$ such that $h(\Sigma)=\Gamma_{\hat{\phi}}$ and
  $$\nabla_{\hat{\phi}}\hat{\phi}(v)=\nabla_\phi \phi(\hat{h}^{-1}(v))+b$$
  for all $v\in V_0$. Thus
  \begin{align*} 
    E_{\hat{h}(D)}(\hat{\phi})
    &=\int_{\hat{h}(D)} \left(\nabla_\phi\phi(\hat{h}^{-1}(v))+b\right)^2 \ud \mu(v)\\
    &=\int_{D} \nabla_\phi\phi(v)^2+ 2b \nabla_\phi\phi(v) +b^2 \ud \mu(v)\\
    &=E_{D}(\phi) + b^2 \mu(D) + 2b\int_{D} \nabla_\phi \phi(v) \ud \mu(v).
  \end{align*}

  Since $\hat{f}=\hat{g}$ except on $\hat{h}(D_0)$,
  $$E_{h(W)}(\Lambda) - E_{h(W)}(h(\Gamma)) = E_{\hat{h}(D_0)}(\hat{f}) - E_{\hat{h}(D_0)}(\hat{g}) = E_{\hat{h}(D)}(\hat{f}) - E_{\hat{h}(D)}(\hat{g}).$$
  By the calculation above, Lemma~\ref{lem:average naf}, and the minimality of $f$,
  \begin{align*}
    E_{\hat{D}}(\hat{f}) & = E_{D}(f) + b^2 \mu(D) + 2b\int_{D} \nabla_f f(v) \ud \mu(v)\\
                         &\le E_{D}(g) + b^2 \mu(D) + 2b\int_{D} \nabla_g g(v) \ud \mu(v)\\
                         &=E_{\hat{D}}(\hat{g}),
  \end{align*}
  so $E_{h(W)}(\Gamma)\le E_{h(W)}(h(\Lambda))$, as desired.
\end{proof}

Furthermore, we can use stretch automorphisms to relate the energy of a surface to its perimeter.
\begin{lemma}\label{lem:stretch limit energy}
  Let $\Lambda=\Gamma_g$ be an intrinsic Lipschitz graph. Let $W\subset \HH$ be a bounded open set and let $D=\Pi(W\cap \Lambda)$. Then, as $r\to \infty$, 
  \begin{equation}\label{eq:stretch limit energy}
    \cS^3\left(s_{r,r^{-1}}(W\cap \Lambda)\right)=r\mu(D) + r^{-3}E_W(\Lambda) + O(r^{-7}),
  \end{equation}
  where the implicit constant is bounded by a function of $\mu(D)$ and the intrinsic Lipschitz constant of $\Lambda$.
\end{lemma}
\begin{proof}
  Let $g_r(x,z) = r^{-1}g(r^{-1}x, z)$, so that $s_{r,r^{-1}}(\Lambda) = \Gamma_{g_r}$.  
  Then
  $$\nabla_{g_r} g_r(v) = r^{-2}\nabla_{g_r} g_r(s_{r,r^{-1}}(v)),$$ and
  \begin{align*}
    \cS^3\left(s_{r,r^{-1}}(W\cap \Lambda)\right) 
    &= \int_{s_{r,r^{-1}(D)}} \sqrt{1+|\nabla_{g_r} g_r|^2} \ud \mu \\ 
    &= \int_{D} r\sqrt{1+r^{-4}|\nabla_{g} g|^2} \ud \mu\\
    &= r\int_{D} 1+\frac{r^{-4}}{2}|\nabla_{g} g|^2+O(r^{-8}|\nabla_{g} g|^4) \ud \mu\\
    &= r\mu(D) + r^{-3}E_W(\Lambda) + O(r^{-7}\mu(D) \sup_{v\in D}|\nabla_{g} g(v)|^4),
  \end{align*}
  as desired.
\end{proof}

This lets us construct energy-minimizing surfaces by stretching $H$--minimal surfaces.  We start with the $H$--minimal surfaces with a singularity along the $x$--axis constructed in \cite{PaulsHMinimal}.  These surfaces were proved to be area-minimizing on any bounded set in \cite[Example~7.1]{ChengHwangYang}, \cite{MSCVnegative}.
\begin{lemma}\label{lem:herringbone}
  Let $a\in \R\setminus \{0\}$ and let $f_a(x,z)=-a \sqrt{z} \sign(z),$
  where $\sign(z)=1,0,$ or $-1$ depending on whether $z$ is positive, zero, or negative.  Let $\Gamma=\Gamma_{f_a}$.  Then $\Gamma$ minimizes energy on any bounded open set $W\subset \HH$.
\end{lemma}
We call these surfaces \emph{herringbone surfaces}, after the arrangement of their horizontal curves. For any $a>0$, $\Gamma=\Gamma_{f_a}$ can be written as the  union of two horizontal rays coming out of each point of the $x$--axis with slope $\pm \frac{a^2}{2}$
\begin{proof}
  Since $\Gamma$ is an area-minimizing surface, if $W$ is a bounded open set and $E$ is a finite-perimeter set such that $E\symdiff \Gamma^+\subset W$, then $\Per_E(W)\ge \Per_{\Gamma^+}(W)$.  In fact, for any $r>0$, we have $s_{r,r^{-1}}(\Gamma)=\Gamma_{f_{r^{-1}a}}$, so $s_{r,r^{-1}}(\Gamma)$ is also area-minimizing.

  Let $\Lambda=\Gamma_g$ be an intrinsic Lipschitz graph such that $\Lambda^+\symdiff \Gamma^+\Subset W$.  The minimality of $s_{r,r^{-1}}(\Gamma)$ implies that for any $r>0$,
  $$\cS^3\left(s_{r,r^{-1}}(W\cap \Gamma)\right)\le \cS^3\left(s_{r,r^{-1}}(W\cap \Lambda)\right).$$
  By Lemma~\ref{lem:stretch limit energy},
  $$\cS^3\left(s_{r,r^{-1}}(W\cap \Lambda)\right)-\cS^3\left(s_{r,r^{-1}}(W\cap \Gamma)\right) = r^{-3}\big(E_W(\Lambda)-E_W(\Gamma)\big) + O(r^{-7}).$$
  This is non-negative for all $r>0$, so as $r$ goes to infinity, we get $E_W(\Gamma)\le E_W(\Lambda)$, as desired.
\end{proof}

One can also prove the minimality of $\Gamma$ using calibrations. In fact, Lemma~\ref{lem:herringbone} is a special case of Theorem~\ref{thm:constructEnergy}; if $K=\{\pm \frac{a^2}{2}\}$, then $\Gamma_{f_a}=\Lambda_{K}$. In the context of minimal surfaces, a calibration is a field of unit vectors with zero divergence. We will adjust this definition for the case of harmonic graphs.

Let $V$ be a Borel horizontal vector field. For an open set $U\subset \HH$ and a finite-perimeter subset $E\subset \HH$, we define the \emph{flux} of $V$ through $\partial E$ as
$$\cF_U(E,V)\eqdef \int_U \langle \nu_E,V\rangle \ud |\partial E|.$$
We say that $V$ is \emph{conservative} on $U$ if $\cF_U(E,V)=0$ for every finite-perimeter subset $E\Subset U$. This implies, in particular, that the distributional divergence of $V$ is zero.

The following proposition gives a calibration condition for proving that a graph is energy minimizing.
\begin{prop}\label{prop:calibrations}
  Let $\Gamma=\Gamma_f$ be an intrinsic Lipschitz graph and let $W\subset \HH$ be an open set.

  Suppose there is a locally bounded Borel function $\tau\from W\to \R$ such that $\tau(p)=\nabla_ff(\Pi(p))$ for almost every $p\in W\cap \Gamma$.  Suppose further that the vector field
  \begin{equation}\label{eq:barM parabola}
    \ovln{M}(p) =-\tau(p) X_{p} + \bigg(1-\frac{\tau(p)^2}{2}\bigg)Y_{p}
  \end{equation}
  is conservative on $W$.  Then $\Gamma$ is energy-minimizing on $W$. 
\end{prop}

The proof of Proposition~\ref{prop:calibrations} is based on Lemma~\ref{lem:stretch limit energy} and the following formula. Recall that for an intrinsic Lipschitz graph $\Gamma=\Gamma_f$, the \emph{unit horizontal normal} at a point $p\in \Gamma$ is given by
$$\nu_\Gamma(p)=\frac{-\nabla_ff(\Pi(p)) X_p + Y_p}{\sqrt{1+\nabla_ff(\Pi(p))^2}}.$$
Let
$$M_{\Gamma}(p)=-\nabla_ff(\Pi(p)) X_p + \left(1 - \frac{\nabla_ff(\Pi(p))^2}{2}\right)Y_p.$$
Then $|M_\Gamma(p)-\nu_\Gamma(p)| \lesssim |\nabla_ff(\Pi(p))|^3$. If $M$ is as in Proposition~\ref{prop:calibrations}, then $M_{\Gamma}(p)=M(p)$ almost everywhere on $W\cap \Gamma$.

Let $W \subset \HH$ be an open set, and let $D=\Pi(W\cap \Gamma)$. Let $V$ be a horizontal vector field $V$. By \eqref{eq:energy area coarse},
\begin{multline}\label{eq:M energy}
  \cF_W(\Gamma^+,M_\Gamma) = \int_{W} \langle M_\Gamma, \nu_{\Gamma} \rangle \ud |\partial \Gamma^+| = \int_{D} \big\langle M_\Gamma\left(\Psi_f(v)\right), - \nabla_ff(v) X_{\Psi_f(v)} + Y_{\Psi_f(v)}\big\rangle \ud \mu(v)\\
  = \int_{D} \nabla_ff(v)^2 + 1 -\frac{1}{2}\nabla_ff(v)^2 \ud \mu(v) = \mu(D) + E_W(\Gamma).
\end{multline}

\begin{proof}[{Proof of Proposition~\ref{prop:calibrations}}]
  Without loss of generality, we may take $W$ to be a bounded open set and $\tau$ to be bounded on $W$. Let $\Lambda=\Gamma_g$ be an entire intrinsic Lipschitz graph such that $\Gamma^+\symdiff \Lambda^+ \Subset W.$
  Let $\Gamma_r=s_{r,r^{-1}}(\Gamma)$ and $\Lambda_r=s_{r,r^{-1}}(\Lambda)$, and let $f_r, g_r\from V_0\to \R$ be the functions such that $\Gamma_{f_r}=\Gamma_r$ and $\Gamma_{g_r}=\Lambda_r$.
  Since $\Gamma^+\symdiff \Lambda^+ \Subset W$, we have $\Pi(W\cap \Lambda)=\Pi(W\cap \Gamma)$; let $D=\Pi(W\cap \Gamma)$ and $D_r=s_{r,r^{-1}}(D)$.

  Let $W_r=s_{r,r^{-1}}(W) $. For $p\in W_r$, let $q=s_{r,r^{-1}}^{-1}(p)$, let 
  $\tau_r(p) = r^{-2} \tau(q),$
  and let $\ovln{M}_r(p) =-\tau_r(p) X_{p} + \big(1-\frac{\tau_r(p)^2}{2}\big)Y_{p}$. 
  For almost every $p\in W_r\cap \Gamma_r$, 
  $$\tau_r(p) = r^{-2} \nabla_f f(\Pi(q)) = \nabla_{f_r} f_r(\Pi(p)),$$
  so $\ovln{M}_r=M_{\Gamma_r}$ almost everywhere on $W_r\cap \Gamma_r$. 
  By \eqref{eq:stretch energy} and \eqref{eq:M energy},
  $$\cF_{W_r}(\Gamma_r^+,\ovln{M}_r) = \mu(D_r) + E_{W_r}(\Gamma_r) = r \mu(D) + r^{-3} E_{W}(\Gamma).$$

  We claim that $\ovln{M}_r$ is conservative. We have
  \begin{align*}
    (s_{r,r^{-1}})_*(\ovln{M}(q)-Y_q) 
    &= - \tau(q) r X_p + \frac{1}{2} \tau(q)^2 r^{-1}Y_p = r^3(\ovln{M}_r(p)-Y_p),
  \end{align*}
  so 
  \begin{equation}\label{eq:barM def}
    \ovln{M}_r=r^{-3} (s_{r,r^{-1}})_*(\ovln{M}-Y) + Y.
  \end{equation}

  Let $E$ be a finite-perimeter set. For any horizontal vector field $V$ and any $|\partial E|$--measurable subset $S\subset \partial^*E$, we have
  $$\int_S \langle \nu_E,V\rangle \ud |\partial E| = \int_{s_{r,r^{-1}}(S)} \langle \nu_{s_{r,r^{-1}}(E)} , (s_{r,r^{-1}})_*(V) \rangle \ud |\partial s_{r,r^{-1}}(E)|.$$
  This follows from a calculation when $S$ is a subset of an $\HH$--regular hypersurface as in \cite{FSSCRectifiability}, and by the Main Theorem of \cite{FSSCRectifiability}, $\partial E$ can be written as a union of compact subsets of $\HH$--regular hypersurfaces, up to a $\cS^3$--null set. Therefore,
  $$\cF_{W_r}(s_{r,r^{-1}}(E), (s_{r,r^{-1}})_*(V))=\cF_{W}(E, V),$$
  and if $E\Subset W$, the conservativity of $M$ implies
  \begin{multline*}
    \cF_{W_r}(s_{r,r^{-1}}(E), \ovln{M}_r)  = r^{-3}\cF_{W_r}(s_{r,r^{-1}}(E), (s_{r,r^{-1}})_*(\ovln{M}-Y)) + \cF_{W_r}(s_{r,r^{-1}}(E), Y) \\
    = r^{-3}\cF_{W}(E, \ovln{M}-Y) + \cF_{W}(E, r Y) = 0.
  \end{multline*}
  Therefore, $\ovln{M}_r$ is conservative.
  
  Now, we calculate $\cF_{W_r}(\Lambda_r^+, \ovln{M}_r)$ and bound $\cS^3(W_r\cap \Lambda_r)$ from below. 
  Let $S=\Lambda_r^+\setminus \Gamma_r^+$ and $T=\Gamma_r^+\setminus \Lambda_r^+$. These sets have finite perimeter by Lemma~\ref{lem:intersection perimeter}, and 
  $\one_{\Lambda_r^+} - \one_{\Gamma_r^+} = \one_{S} -  \one_{T}.$
  Taking distributional gradients, equation \eqref{eq:dist grad} implies
  $$\nu_{\Lambda_r^+} |\partial \Lambda_r^+| - \nu_{\Gamma_r^+} |\partial \Gamma_r^+| = \nu_{S} |\partial S| - \nu_{T} |\partial T|$$
  as vector-valued Radon measures, and thus
  \begin{multline*}
    \cF_{W_r}(\Lambda_r^+, \ovln{M}_r) = \cF_{W_r}(\Gamma_r^+, \ovln{M}_r) + \cF_{W_r}(S, \ovln{M}_r) - \cF_{W_r}(T, \ovln{M}_r)\\
    = \cF_{W_r}(\Gamma_r^+, \ovln{M}_r) = r \mu(D) + r^{-3} E_{W}(\Gamma).
  \end{multline*}
  It follows that
  \begin{equation}\label{eq:bound flux lambda}
    \cS^3(W_r\cap \Lambda_r) \ge \frac{\cF_{W_r}(\Lambda_r^+, \ovln{M}_r)}{\sup_{q\in W_r} |\ovln{M}_r(q)|} =\frac{r \mu(D) + r^{-3} E_{W}(\Gamma)}{\sup_{q\in W_r} |\ovln{M}_r(q)|}.
  \end{equation}
  Let $t=\sup_{p\in W} |\tau(p)|$. By \eqref{eq:barM parabola}, for $q\in W_r$,
  $$|\ovln{M}_r(q)| = \sqrt{\tau_r(q)^2 + 1 - \tau_r(q)^2 + \frac{\tau_r(q)^4}{4}} 
  = 1 + O(r^{-8}t^4),$$
  so for sufficiently large $r$,
  \begin{equation}\label{eq:calibration energy}
    \cS^3(W_r\cap \Lambda_r) \ge \frac{r \mu(D) + r^{-3} E_{W}(\Gamma)}{1 + O(r^{-8}t^4)} = r \mu(D) + r^{-3} E_{W}(\Gamma) + O(r^{-7}),
  \end{equation}
  where the last implicit constant depends on $D$, $t$, and $\Gamma$.
  By Lemma~\ref{lem:stretch limit energy},
  \begin{equation}\label{eq:backref sle}
    \cS^3(W_r\cap \Lambda_r)=r\mu(D) + r^{-3}E_W(\Lambda) + O(r^{-7}).
  \end{equation}
  Comparing \eqref{eq:calibration energy} and \eqref{eq:backref sle}, we find $E_W(\Lambda)\ge E_W(\Gamma)$, as desired.
\end{proof}

Thus, we can construct energy minimizers by constructing conservative vector fields of the form \eqref{eq:barM parabola}.

\begin{example}\label{ex:ruled}
  We say that a \emph{ruled surface} is a surface in $\HH$ foliated by horizontal line segments. Let  $\Sigma$ be a $C^2$ ruled $Z$--graph, that is, a ruled surface of the form $\Sigma=\{(x,y,g(x,y))\mid x,y\in \Omega\}$ where $\Omega\subset \R^2$ is open and $g\in C^2(\Omega)$.  The results of \cite{ChengHwangYangMalchiodi} show that $\Sigma$ is an area minimizer.

  Suppose that the slopes of the horizontal lines making up $\Sigma$ are bounded, so that $\Sigma=\Gamma_f$ is locally intrinsic Lipschitz. We claim that $\Sigma$ is an energy minimizer. For every $q=(x,y) \in \Omega$, there is a unique point $p=(x,y,g(x,y))\in \Sigma$ and a unique maximal horizontal line segment $\widetilde{L}_{p}$ through $p$ that is contained in $\Sigma$. Let $L_{q}=\pi(\widetilde{L}_{p})$ and let $\sigma(q)=\nabla_ff(\Pi(p))$ be the slope of $L_{q}$. For any $q'\in L_q$, the uniqueness of $L_q$ implies that $L_{q'}=L_q$, so $\sigma$ is constant along lines of the form $L_{q}$. 

  Let $W=\pi^{-1}(\Omega)$ and for $w\in W$, let $\tau(w)=\sigma(\pi(w))$. Then $\tau(p)=\nabla_ff(\Pi(p))$ for every $p\in\Sigma$.
  Define $\ovln{M}$ as in \eqref{eq:barM parabola}. This is a $C^1$ vector field. For any $p\in W$ and $q=\pi(p)$, 
  $$\div_\HH \ovln{M}(p) = -X[\tau](p) - \tau(p) Y[\tau](p) = -(\partial_x +\sigma(q) \partial_y)[\sigma](q).$$
  The vector $\partial_x +\sigma(q) \partial_y$ is tangent to $L_q$ and $\sigma$ is constant along $L_q$, so $\div_\HH \ovln{M}(p)=0$, i.e., $\ovln{M}$ is conservative. By Proposition~\ref{prop:calibrations}, $\Sigma$ is energy-minimizing.
\end{example}

We can construct more examples by gluing together $C^1$ fields. Nicolussi Golo and Ritoré \cite{GR2020} gave a construction of piecewise $C^1$ vector fields with zero distributional divergence, and a similar construction produces conservative vector fields.
\begin{prop}[{see \cite[Prop.\ 2.5]{GR2020}}] \label{prop:GR conserv}
  Let $U\subset \HH$ be an open set. Let $\{\Omega_j\}_j$ be a family of disjoint subsets of $U$ with locally finite perimeter such that $\cS^3(\partial \Omega_j\setminus \partial^*\Omega_j)=0$, $U=\bigsqcup_j \Omega_j$, and $\{\overline{\Omega_j}\}_j$ has locally finite multiplicity.
  For each $j$, let $V_j\in C^1(\overline{\Omega_j};\mathsf{A})$. Let
  $$V=\sum_j V_j \one_{\Omega_j}.$$
  Suppose that for each $j$, $\div_\HH V_j=0$ on $\Omega_j$ and for $\cS^3$--a.e.\ $p\in \partial \Omega_j$,
  \begin{equation}\label{eq:conserv discont}
    \langle \nu_{\Omega_j}(p), V(p)\rangle=\langle \nu_{\Omega_j}(p), V_j(p)\rangle.
  \end{equation}
  Then $V$ is conservative.
\end{prop}

\begin{proof}
  Let $E\Subset U$ be a finite-perimeter set and for each $j$, let $E_j=E\cap \Omega_j$. By Lemma~\ref{lem:intersection perimeter}, $E_j$ has finite perimeter.
  We claim that $\cF_U(E_j,V)=\cF_U(E_j,V_j) = 0$. It suffices to show that \begin{equation}\label{eq:equal flux cond}
    \langle \nu_{E_j}(p), V(p)\rangle=\langle \nu_{E_j}(p), V_j(p)\rangle
  \end{equation}
  for $\cS^3$--a.e.\ $p\in \partial^*E_j$.

  We have $\partial^*E_j\subset \overline{E_j}\subset \overline{\Omega_j}$, so for all $p\in \partial^*E_j$ except a $\cS^3$--null set, we have $p\in \Omega_j$ or $p\in \partial^*\Omega_j$. If $p\in \Omega_j$, then $V(p)=V_j(p)$, and \eqref{eq:equal flux cond} holds.

  Otherwise, $p\in \partial^*E_j \cap \partial^*\Omega_j$. Lemma~\ref{lem:intersection perimeter} implies $\nu_{E_j}(p)=\nu_{\Omega_j}(p)$, and \eqref{eq:conserv discont} implies that $\langle \nu_{E_j}(p), V(p)\rangle=\langle \nu_{E_j}(p), V_j(p)\rangle$
  for $\cS^3$--a.e.\ $p\in \partial^*E_j\cap \partial^* \Omega_j$. Thus, by \eqref{eq:div thm},
  $$\cF_U(E_j,V)=\cF_U(E_j,V_j) = \int_{E_j} \div_\HH V_j\ud \cS^4=0.$$
  By \eqref{eq:dist grad}, 
  $$\nu_E|\partial E|=\nabla \one_E=\sum_j \nabla \one_{E_j}=\sum_j \nu_E|\partial E|$$
  as vector-valued Radon measures, so 
  $$\cF_U(E,V)=\int_U \langle V, \nu_E \rangle \ud|\partial E| = \sum_j \int_U \langle V, \nu_{E_j} \rangle \ud|\partial E|_j = \sum_j \cF_U(E_j,V)=0.$$
\end{proof}

When $V=\ovln{M}$ is as in Proposition~\ref{prop:calibrations}, Condition \eqref{eq:conserv discont} can be written in terms of $\tau$. Let $\Omega^+,\Omega^-\subset \HH$ be locally finite-perimeter sets and let $p\in \partial^* \Omega^-\cap \partial^* \Omega^+$. Then
$$\nu_{\Omega^+}(p)=-\nu_{\Omega^-}(p)=\frac{-\sigma X + Y}{\sqrt{1+\sigma^2}}$$
for some $\sigma\in \R$. Let $\tau^\pm\in C^1(\overline{\Omega^\pm})$ and let
$$V^\pm =-\tau^\pm X + \bigg(1-\frac{(\tau^\pm)^2}{2}\bigg)Y.$$
Then 
\begin{multline*}
  \langle -\sigma X + Y, V^+(p)-V^-(p)\rangle
   = \sigma (\tau^+(p)-\tau^-(p)) - \frac{1}{2}\left(\tau^+(p)^2 -\tau^-(p)^2\right)\\
   = (\tau^+(p)-\tau^-(p))\bigg(\sigma  - \frac{\tau^+(p)^2 + \tau^-(p)^2}{2}\bigg).
\end{multline*}
That is, \eqref{eq:conserv discont} is satisfied at $p$ if and only if $\tau^+(p)=\tau^-(p)$ or
\begin{equation}
\label{eq:tau conserv}
  \sigma = \frac{\tau^+(p)+\tau^-(p)}{2}.
\end{equation}
This is analogous to the equal-angle condition studied in \cite{ChengHwangYang}.

We can thus construct energy-minimizing surfaces by gluing ruled surfaces along singularities. We will need a criterion that ensures that a union of horizontal curves is an intrinsic Lipschitz graph.
\begin{lemma}\label{lem:blarg}
  Let $c>0$ and let
  $$C=\Big\{q\in \HH\mid |y(q)|> \max\big\{ 4c|x(q)|, \sqrt{32c |z(q)|}\big\}\Big\}$$
  be a scale-invariant double cone in $\HH$. 
  For $i=1,2$, let $\alpha_i\from [0,t_i]\to \HH$ be unit $x$--speed horizontal curves such that $\alpha_1(0)=\alpha_2(0)$ and $\Lip(y\circ\alpha_i)<c$. Suppose further that the projections $\pi\circ \alpha_i$ do not cross, i.e., $y(\alpha_1(t))\le y(\alpha_2(t))$ for all $t\in [0,\min(t_1,t_2)]$ or $y(\alpha_1(t))\ge y(\alpha_2(t))$ for all $t\in [0,\min(t_1,t_2)]$. Let $p_i=\alpha_i(t_i)$.  Then $p_1^{-1}p_2\not \in C$.
\end{lemma}
\begin{proof}
  Note that $C^{-1}=C$, so $p_1^{-1}p_2\in C$ if and only if $p_2^{-1}p_1\not \in C$. After possibly switching $\alpha_1$ and $\alpha_2$, we may suppose that $t_1\le t_2$.  Since $s_{1,-1}(C)=C$, after possibly replacing $\alpha_i$ by $s_{1,-1}(\alpha_i)$, we may suppose that $y(\alpha_1(t))\le y(\alpha_2(t))$ for all $t\in [0,t_1]$.
  
  We translate so that $p_1=\mathbf{0}$.  It suffices to show that $|y(p_2)|\le 4c|x(p_2)|$ or $|y(p_2)|\le \sqrt{8c|z(p_2)|}$.  Suppose that $|y(p_2)|>4c|x(p_2)|$. We claim that $z(p_2)\le -\frac{y(p_2)^2}{8c}$.

  Since $|y(p_2)|>4cx(p_2)$ and $\Lip(y\circ\alpha_2)<c$, we have $|y(\alpha_2(t_1))|> 3cx(p_2).$ By hypothesis, $y(\alpha_1(t_1))=0\le y(\alpha_2(t_1))$, so
  $$y(\alpha_2(0)) \ge y(\alpha_2(t_2))-cx(p_2) > 3cx(p_2)$$
  and $y(p_2)>4cx(p_2)$. In particular, 
  \begin{equation}\label{eq:yalpha2}
    y(\alpha_2(t))\ge \frac{y(p_2)}{2} \qquad \text{ for } t_1\le t\le t_2.
  \end{equation}

  Since $\alpha_1$ and $\alpha_2$ are horizontal, \eqref{eq:horiz proj} implies $(z\circ \Pi\circ \alpha_i)'=-y\circ\alpha_i$ and
  \begin{align*}
    z(\Pi(p_1)) - z(\Pi(p_2)) = - z(\Pi(p_2)) 
    & = \int_0^{t_1} - y(\alpha_1(t))\ud t - \int_0^{t_2}-y(\alpha_2(t))\ud t\\
    & = \int_0^{t_1} y(\alpha_2(t)) - y(\alpha_1(t)) \ud t + \int_{t_1}^{t_2} y(\alpha_2(t)) \ud t \\
    & \stackrel{\eqref{eq:yalpha2}}{=} \int_0^{t_1} y(\alpha_2(t)) - y(\alpha_1(t)) \ud t +  \frac{x(p_2)y(p_2)}{2}
  \end{align*}
  Since $\Lip(y\circ\alpha_i)<c$, for $t_1-\frac{y(p_2)}{8c}\le t\le t_1$, 
  $$y(\alpha_2(t)) - y(\alpha_1(t))\ge y(\alpha_2(t_1)) - y(\alpha_1(t_1)) - 2c \frac{y(p_2)}{8c} \stackrel{\eqref{eq:yalpha2}}{\ge} \frac{y(p_2)}{4}.$$
  Thus
  $$-z(\Pi(p_2))\ge \int_{t_1-\frac{y(p_2)}{8c}}^{t_1} y(\alpha_2(t)) - y(\alpha_1(t)) \ud t + \frac{x(p_2) y(p_2)}{2} \ge \frac{y(p_2)^2}{32c} + \frac{x(p_2) y(p_2)}{2}.$$
  By \eqref{eq:def Pi},
  $$z(p_2) = z(\Pi(p_2)) + \frac{x(p_2) y(p_2)}{2} \ge \frac{y(p_2)^2}{32c},$$
  so $|y(p_2)|\le \sqrt{32c |z(p_2)|}$, as desired.
\end{proof}

\begin{remark}\label{rem:lipgraphs}
  We use Lemma~\ref{lem:blarg} to show that certain $Z$--graphs are intrinsic Lipschitz graphs. A sufficiently regular $Z$--graph $\Gamma$ can be written as a union of horizontal curves. In the examples we will consider, there will be a basepoint $p_0$ such that for every $p\in \Gamma$, there is a horizontal curve $\gamma_p$ from $p_0$ to $p$. If the projections $\pi\circ \gamma_p$ are the graphs of Lipschitz functions and no two such projections cross, then for any $p,q\in \Gamma$, applying Lemma~\ref{lem:blarg} to $\gamma_p$ and $\gamma_q$ implies that $p^{-1}q\not\in C$. Since $C$ is a scale-invariant open double cone containing $Y$, this implies that $\Gamma$ is an intrinsic Lipschitz graph.
\end{remark}

\begin{example}\label{ex:flex}
  Let $\gamma\from \R\to \HH$ be a smooth horizontal curve with unit $x$--speed. Let $\sigma(s)=(y\circ \gamma)'(s)$ be the slope of $\gamma$ and let $\delta\in C^\infty(\R)$, $\delta>0$. Let $\rho\from \R^2\to \HH$,
  $$\rho(s,t)=\begin{cases}
    \gamma(s)\cdot \big(X+\left(\sigma(s)+\delta(s)\right) Y\big)^{|t|} & t \ge 0 \\
    \gamma(s)\cdot \big(X+\left(\sigma(s)-\delta(s)\right) Y\big)^{|t|} & t < 0.
  \end{cases}
  $$

  Let $I=(a,b)\subset \R$ be an open interval and choose $0<\epsilon\le \infty$ small enough that $\Sigma=\rho(I\times(-\epsilon,\epsilon))$ is a $Z$--graph. Let $p_0=\gamma(a)$. For any point $p=\rho(s,t)$, there is a horizontal curve in $\Sigma$ from $p_0$ to $p$ which consists of the segment $\gamma([a,s])$ concatenated with a horizontal ray of slope $\sigma(s)\pm \delta(s)$ originating at $\gamma(s)$. These curves all project to Lipschitz graphs in $\R^2$, and no two such graphs cross, so by Remark~\ref{rem:lipgraphs}, $\Sigma$ is an intrinsic Lipschitz graph.

  We construct a calibration for $\Sigma$ as follows. Let $\Omega=\pi(\Sigma)$, and let $W=\pi^{-1}(\Omega)$. For any $p\in W$, there is a unique point $(s,t)\in I\times(-\epsilon,\epsilon)$ such that $\pi(\rho(s,t))=\pi(p)$. Define
  $$\tau(p)=\begin{cases}
    \sigma(s)+\delta(s) & t\ge 0\\
    \sigma(s)-\delta(s) & t < 0
  \end{cases}$$
  so that $\tau(p)=\nabla_ff(\Pi(p))$ for every $p\in \Sigma\setminus \gamma$, 
  and let $\ovln{M}=-\tau X + (1-\frac{\tau^2}{2})Y$.

  The surface $\gamma\langle Z\rangle$ cuts $W$ into two halves $W^+$ and $W^-$ such that $\tau$ is smooth on each half. Each half $W^\pm\cap \Sigma$ is a smooth ruled $Z$--graph, so by Example~\ref{ex:ruled}, $\div_{\HH}\ovln{M}=0$ on $W^\pm$. If we extend $\tau$ to $\tau^\pm\in C^1(\overline{W^\pm})$ by continuity, it satisfies \eqref{eq:tau conserv} everywhere along the boundary, so by Proposition~\ref{prop:GR conserv}, $\ovln{M}$ is conservative. By Proposition~\ref{prop:calibrations}, $\Sigma$ is energy-minimizing on $W$.
\end{example}

This lets us prove Theorem~\ref{thm:constructEnergy}.
\begin{proof}[Proof of Theorem~\ref{thm:constructEnergy}]
  Let $\alpha>0$ and let $I_1,I_2,\dots \subset [-\alpha, \alpha]$ be a collection of disjoint nonempty open intervals. Let $I_i=(a_i,b_i)$, let $m_i=\frac{a_i+b_i}{2}$, and let $\delta_i=\frac{b_i-a_i}{2}$. Let $K=[-\alpha, \alpha]\setminus \bigcup I_i$ and let $\Lambda_{K}\subset \HH$ be as in the statement of the theorem. Let $R_0$ be the negative $x$--axis, and for each $i\ge 1$, let $R_i$ be a positive horizontal ray from the origin with slope $m_i$.

  We first decompose $\R^2$ into wedges.  Let
  $$W_0=\left\{(x,y)\in \R^2\mid |y|>\alpha x\right\},$$
  $$W_i=\left\{(x,y)\in \R^2\mid x>0, \frac{y}{x} \in I_i\right\},$$
  and
  $$P_K=\left\{(x,y) \in \R^2 \mid x> 0, \left|\frac{y}{x}\right| \in K\right\}\cup \{(0,0)\}.$$
  These sets are disjoint and their union is $\R^2$.

  Every point $p\in \Lambda_K$ is the endpoint of a horizontal curve. That is, for any $p\in \Lambda_K$, there is a unique horizontal curve $\gamma_p\from (-\infty, x(p)] \to \Lambda_K$ with unit $x$--speed such that $\gamma_p(x(p))=p$. These curves are illustrated in Figure~\ref{fig:branch}. The shape of the curve depends on which wedge contains $\pi(p)$. When $\pi(p)\in W_0$, $\gamma_p$ travels along the negative $x$--axis, then along a horizontal ray of slope $\pm\alpha$. When $\pi(p)\in W_i$, $\gamma_p$ travels along the negative $x$--axis to $\mathbf{0}$, along a horizontal ray of slope $m_i$, then along a horizontal ray of slope $m_i\pm \delta_i$. When $\pi(p)\in P_K$, $\gamma_p$ travels along the negative $x$--axis to $\mathbf{0}$, then along a horizontal ray of slope $k$, for some $k\in K$.

  For any two points $p,q\in \Lambda_K$, the curves $\gamma_p$ and $\gamma_q$ satisfy Lemma~\ref{lem:blarg} with $\max\{\Lip(y\circ\gamma_p), \Lip(y\circ\gamma_q)\}\le \alpha$, so there is a scale-invariant double cone such that $p\not \in qC$. It follows that $\Lambda_K$ is an intrinsic Lipschitz graph; in fact it is an entire intrinsic Lipschitz graph. By construction, it is also an entire $Z$--graph. Let $h\from V_0\to \R$ be the function such that $\Lambda_K=\Gamma_h$.
  
  Next, we construct a function $\tau_K$ and a corresponding calibration $\ovln{M}_K$ to show that $\Lambda_K$ is energy-minimizing. Since $\Lambda_K$ is a $Z$--graph, we choose a $\tau_K$ that is constant on vertical lines.  For $i>0$, let
  $$W_i^+=\{(x,y)\in W_i \mid y\ge m_ix\},$$
  $$W_i^-=\{(x,y)\in W_i \mid y< m_i x\}.$$
  Let
  $$\tau_K(x,y,z)=
  \begin{cases}
    m_i\pm \delta_i & (x,y)\in W_i^\pm \\
    \alpha & (x,y)\in W_0, y\ge 0 \\
    -\alpha & (x,y)\in W_0, y < 0 \\
    \frac{y}{x} & (x,y)\in P_K.
  \end{cases}$$

 
  We claim that $\ovln{M}_K =-\tau_K X + (1-\frac{\tau_K}{2})Y$ is conservative and that $\tau(p)=\nabla_{h}h(\Pi(p))$ almost everywhere on $\Lambda_K$. Let $K_n=[-\alpha,\alpha]\setminus (I_1\cup I_2\cup \dots \cup I_n)$ and define $P_{K_n}$, $\tau_{K_n}$ and $\ovln{M}_{K_n}$ as above. 
  Then $\HH$ is the union of finitely many sets
  $$\HH=P_{K_n}\cup \bigcup_{i=0}^n W_i^+\cup \bigcup_{i=0}^n W_i^-,$$
  each with locally finite perimeter. On each of these sets, $\ovln{M}_{K_n}$ is $C^1$ and $\div_\HH \ovln{M}_{K_n}=0$. Further, $\tau_{K_n}$ is continuous except along the vertical half-planes separating $W_i^+$ from $W_i^-$.  These half-planes have slope $m_i$ (taking $m_0=0$), and $\tau_{K_n}$ takes values $m_i\pm \delta_i$ above and below the central ray (taking $\delta_0=\alpha$), so $\tau_{K_n}$ satisfies \eqref{eq:tau conserv}. By Proposition~\ref{prop:GR conserv}, $\ovln{M}_{K_n}$ is conservative.  Since $\ovln{M}_{K}$ is the uniform limit of the $\ovln{M}_{K_n}$'s, $\ovln{M}_{K}$ is also conservative.


  We calculate $\nabla_hh$ by applying Theorem~1.2 of \cite{BiCaSC}. This theorem implies that there is a $\cS^3$--null subset $R \subset \Lambda_K$ such that for any $p\in\Lambda_K\setminus S$ and any unit $x$--speed horizontal curve $\beta=(\beta_x,\beta_y,\beta_z)\from (-\epsilon,\epsilon)\to \Lambda_K$ such that $\beta(0)=p$, if $\beta_y'$ is differentiable, then $\nabla_hh(\Pi(p))=\beta_y'(0)$.\footnote{The original theorem is stated in terms of integral curves of $\nabla_h$. This version is obtained by applying the original theorem to $\Pi\circ \beta$, which is an integral curve of $\nabla_h$.}
  
  Let $p\in \Lambda_K\setminus S$ be a point that does not lie on the negative $x$--axis or on the countably many horizontal rays bisecting the $W_i$'s. Then $\Lambda_K$ contains a horizontal ray of slope $\tau_K(p)$ through $p$, so
  $\nabla_hh(\Pi(p)) =\tau_K(p),$ as desired. Proposition~\ref{prop:calibrations} then implies that $\Lambda_K$ is energy-minimizing.
\end{proof}

Finally, we show that these energy-minimizing graphs can be written as limits of the stretched $H$--minimal graphs $\Sigma_K$ constructed in \cite{GR2020} (see Theorem~\ref{thm:constructArea}).

\begin{proof}[Proof of Proposition~\ref{prop:stretchLimits}]
  Let $\alpha$ and $K=[-\alpha, \alpha]\setminus \bigcup I_i$ be as in Theorem~\ref{thm:constructEnergy}. Let $m_i$ be the midpoint of $I_i$, and for $n>\sqrt{\alpha}$, let
  $K_n= n^{-2} K$. Let $\Sigma_{K_n}$ as in Theorem~\ref{thm:constructArea}, and let $S_n=s_{n^{-1},n}(\Sigma_{K_n})$. We claim that the $S_n$'s are intrinsic Lipschitz graphs and that $\one_{S_n^+}\to \one_{\Lambda_K^+}$ in $L_1^{\mathrm{loc}}(\HH)$.

  First, note that $\Sigma_{K_n}$ can be written as a union of horizontal rays in the direction $\pm n^{-2}\alpha\in S^1$, rays in the direction of some $v\in K_n$, and rays in the direction of $n^{-2}m_{i}$. If $R_n$ is a ray with $\angle(R,X)=n^{-2}\theta$, then $s_{n^{-1},n}(R_n)$ is a ray with slope $n^2\tan^{-1}(n^{-2}\theta)\to \theta$ as $n\to\infty$. It follows that if $n$ is sufficiently large, then $S_n$ is a union of horizontal rays with slopes between $-2\alpha$ and $2\alpha$. In fact, for every $p\in S_n$, there is a horizontal curve $\beta^n_p\from (-\infty, x(p)] \to S_n$ such that $x(\beta^n_p(t))=t$ and $\beta^n_p(x(p))=p$.
  Any two curves $\beta^n_p$ and $\beta^n_q$ satisfy Lemma~\ref{lem:blarg}, so the $S_n$'s are intrinsic Lipschitz graphs with uniform Lipschitz constant.

  The $\R$--trees formed by the horizontal curves in $\Lambda_K$ and $S_n$ are isomorphic. That is, there are homomorphisms $h_n\from \Lambda_K \to  S_n$ that send each horizontal curve in $\Lambda_K$ to a horizontal curve in $S_n$, and these can be chosen so that for any $p\in \Lambda_K$, the uniform limit $\lim_{n\to\infty}\beta^n_{h_n(p)}$ is a horizontal curve in $\Lambda_K$ ending at $p$.
  Consequently, $h_n\to \id_{\Lambda_K}$ uniformly on bounded sets, and since the $S_n$'s are uniformly intrinsic Lipschitz, we have $\one_{S_n^+}\to \one_{\Lambda_K^+}$ in $L_1^{\mathrm{loc}}(\HH)$.
\end{proof}

\section{Contact harmonic graphs and first variation formulas}\label{sec:graphs and variations}

In this section, we propose an intrinsic analogue of harmonic functions and discuss some applications and limitations. It is natural to define a harmonic intrinsic graph in $\HH$ as a critical point of the \emph{intrinsic Dirichlet energy}
$$E_U(\Gamma)=\frac{1}{2}\int_U (\nabla_f f)^2\ud\mu.$$
It is not clear, however, how to choose an appropriate class of variations of $\Gamma$.

When $\Gamma$ is smooth or piecewise smooth, we may consider smooth perturbations of $\Gamma$. Critical points with respect to smooth perturbations have vanishing horizontal mean curvature.
\begin{prop}[First variation formula for smooth graphs]\label{prop:fvf range}
  Let $U\subset V_0$ be a bounded open set.  Let $f\in C^\infty(U)$, $h\in C^\infty_c(U)$. For $t\in (-\epsilon, \epsilon)$, let $f_t=f+th$ and $\Gamma_t=\Gamma_{f_t}$. Then
  $$\frac{\ud}{\ud t} E_U(f_{t}) = -\int_{U} \nabla_f^2 f\cdot h  \ud \mu.$$
\end{prop}
\begin{proof}
  First, note that
  \begin{align*}
    \pd{}{t} \nabla_{f_t}f_t
    &= \pd{}{t}\left[\partial_x[f+th] - (f+th) \partial_z[f+th]\right]\\
    &= \partial_x h - h \cdot \partial_z[f+th] - (f+th) \partial_{z}h.
  \end{align*}
  At $t=0$, 
  $$\left.\pd{}{t} \nabla_{f_t}f_t\right|_{t=0} = \partial_{x} h - h \pd{f}{z} - f \partial_z h = \nabla_f h - h \partial_z f.$$

  Therefore,
  \begin{align*}
    \left. \frac{\ud}{\ud t} E(\Gamma_{f_t}) \right|_{t=0} 
    &= \int_U \nabla_f f\cdot \left(\nabla_f h - h \cdot \partial_z f\right) \ud\mu= 
     \int_U \nabla_f f \cdot \nabla_f h - \nabla_f f \cdot h \cdot \partial_z f \ud\mu.
  \end{align*}
  By Corollary~\ref{cor:integration by parts}, with $g=\nabla_ff$ and $h=h$,
  \begin{align*}
    \left. \frac{\ud}{\ud t} E(\Gamma_{f_t}) \right|_{t=0} 
    &= - \int_U \nabla^2_ff \cdot h \ud\mu.
  \end{align*}
\end{proof}

Smooth energy-minimizing graphs are thus foliated by horizontal lines.
\begin{prop}\label{prop:least energy foliated}
  Let $U\subset V_0$ be a bounded open set and let $f\from U\to \R$ be a smooth function which is energy-minimizing on $U$.  Then $\nabla_f^2f=0$ on $U$, and for every $p\in \Gamma_f$, there is a horizontal line segment $L_p$ with endpoints in $\partial \Gamma_f$ such that $p\subset L_p\subset \Gamma_f$. 
\end{prop}
\begin{proof}
  By minimality, for any $h\in C^\infty_c(U)$ and any $t\in \R$, we have $E_U(\Gamma_{f})\le E_U(\Gamma_{f+th})$, so by Proposition~\ref{prop:fvf range}, 
  $$\left.\frac{\ud}{\ud t} E(\Gamma_{f+th}) \right|_{t=0} =- \int_U \nabla^2_ff \cdot h \ud\mu = 0$$
  and $\nabla^2_ff=0$ on $U$.

  Let $p\in \Gamma$ and $u=\Pi(p)$ and suppose that $u\in U$.  By the smoothness of $f$, there is a unique maximal integral curve of $\nabla_f$ through $p$, i.e., a curve $\gamma=(\gamma_x,\gamma_y,\gamma_z) \from I\to U$ such that $\gamma(0)=\Pi(p)$ and $\gamma'(t)=\partial_x - f(\gamma(t))\partial_z$.  Then $\lambda=\Psi_f\circ \gamma=\gamma(t)Y^{f(\gamma(t))}=: (\lambda_x,\lambda_y,\lambda_z)$ is a horizontal curve in $\Gamma$ satisfying $\lambda_x'=\gamma_x'=1$ and 
  $$\lambda_y''(t)=(f\circ \gamma)''(t)=\nabla_f^2f(\gamma(t))=0$$
  for all $t$. That is, $L_p=\lambda$ is a horizontal line segment containing $p$, as desired.
\end{proof}

Unfortunately, when $f$ is intrinsic Lipschitz, the perturbation $f+th$ need not be intrinsic Lipschitz. As an alternative, one can apply contact variations, like those considered in \cite{FogMonVitVar, Golo}. A \emph{contact diffeomorphism} of $\HH$ is a diffeomorphism $\HH\to \HH$ that sends horizontal vectors to horizontal vectors. A \emph{contact flow} is a one-parameter flow of such diffeomorphisms. Contact flows are generated by contact vector fields, and any contact vector field is determined by a potential function \cite[Sec.\ 5]{KorReiFoundations}; for any smooth function $\psi\in C^\infty(\HH)$, the corresponding contact vector field is given by
\begin{equation}\label{eq:contact field}
  V_\psi = (Y\psi)X - (X\psi)Y + \psi Z.
\end{equation}
We are particularly interested in contact diffeomorphisms and flows that send intrinsic graphs to intrinsic graphs.
\begin{defn}\label{defn:contact graph}
  A \emph{contact graph diffeomorphism} is a contact diffeomorphism that sends cosets of $\langle Y \rangle$ to cosets of $\langle Y \rangle$.  Flows of contact graph diffeomorphisms are generated by contact vector fields $V_\psi$ whose $X$--component is constant on each coset of $\langle Y \rangle$, i.e., fields generated by potentials satisfying $YY\psi=0$.
  
  Let $\phi_u(\Gamma), u\in (-\epsilon, \epsilon)$ be the flow generated by a contact vector field $V_\psi$ whose potential satisfies $YY\psi=0$. For an intrinsic graph $\Gamma$, we call the family of surfaces of the form $\phi_u(\Gamma), u\in (-\epsilon, \epsilon)$ a \emph{contact graph variation}. 
\end{defn}

Section~4 of \cite{Golo} describes these diffeomorphisms in terms of diffeomorphisms from $V_0$ to $V_0$. We can identify the coset space $\HH/\langle Y\rangle$ with $V_0$ by the projection $\Pi$. A contact graph diffeomorphism $\phi$ then determines a diffeomorphism $\bar{\phi}\from V_0\to V_0$ such that $\bar{\phi}(v)=\Pi(\phi(v))$ for all $v\in V_0$.

When $\bar{\phi}$ is $C^1$--close to the identity, it sends characteristic curves of an intrinsic Lipschitz graph  $\Gamma$ to characteristic curves of $\phi(\Gamma)$. That is, let $\Gamma=\Gamma_f$ be an intrinsic Lipschitz graph and suppose that $\nabla_f[x\circ \bar{\phi}](v)>0$ for all $v\in V_0$. Let $\gamma\from \R\to V_0$ be a characteristic curve of $\Gamma$, parametrized with unit $x$--speed. Then $\lambda=\Psi_f\circ \gamma$ is a horizontal curve in $\Gamma$, and $\phi\circ \lambda$ is a horizontal curve in $\phi(\Gamma)$ such that
$$(x\circ \phi \circ \lambda)'(t) = (x\circ \bar{\phi} \circ \gamma)'(t) = \nabla_f[x \circ \bar{\phi}](\gamma(t)) >0.$$
Thus the $x$--coordinate of $\Pi\circ \phi \circ \lambda = \bar{\phi}\circ \gamma$ is increasing, and it is a characteristic curve of $\phi(\Gamma)$. 

We propose the following definition. 
\begin{defn}  
  Let $U\subset V_0$ be an open set, let $f\from U\to \R$ be an intrinsic Lipschitz function, and let $\Gamma=\Gamma_f$. We say that $\Gamma$ is \emph{contact harmonic} if $f$ is a critical point of $E_U$ among contact graph variations with potentials whose support lies in $\Pi^{-1}(U)$.
\end{defn}
In the rest of this section, we will prove some first variation formulas for the intrinsic Dirichlet energy and use them to characterize contact harmonic graphs.

\subsection{First variation for intrinsic Lipschitz graphs}\label{sec:fvf lip}

We first prove a formula for the variation of the energy of an intrinsic Lipschitz graph under a contact flow. 
We start by considering the smooth case, as in \cite{FogMonVitVar}.

\begin{thm}[First variation formula for contact variations]\label{thm:fvf domain smooth}
  Let $\Gamma=\Gamma_f$ be an intrinsic graph of a smooth function. Let $\psi\in C^\infty(\HH)$ be a potential such that $YY\psi=0$, let $V_\psi= (Y\psi)X - (X\psi)Y + \psi Z$, and let $\phi_t\from \HH\to \HH$, $t\in [-\epsilon,\epsilon]$ be the flow of $V_\psi$.  Let $\Gamma_t=\phi_t(\Gamma)$.  

  Let $\bar{\phi}_t\from V_0\to V_0$, $\bar{\phi}_t(v)=\Pi(\phi_t(v))$ be the corresponding family of diffeomorphisms of $V_0$.  Let $W$ be the vector field on $V_0$ generating the $\bar{\phi}_t$'s, i.e., 
  $$W(v) = \Pi_*(V_\psi(v)) = (Y\psi(v),0,\psi(v))$$
  for all $v\in V_0$.  Let $w_1=x\circ W$, $w_2=(x\circ W)\cdot f + z\circ W$, so that $W= w_1 \nabla_f + w_2\partial_z.$
  
  Let $U\subset V_0$ be a bounded open subset. There are $\epsilon_0=\epsilon_0 (U, \psi, \|f\|_{L_\infty(U)})>0$ and $C=C(U, \psi, \|f\|_{L_\infty(U)})>0$ such that 
  \begin{equation}\label{eq:fvf domain full}
    |E_{\bar{\phi}_t(U)}(f_t) - E_{U}(f) - (A_1+A_2) t|\le C \left(E_{U}(f) +\mu(U)\right) t^2
  \end{equation}
  for all $t\in [-\epsilon_0,\epsilon_0]$,
  where $A_1$ and $A_2$ depend on $w_1$ and $w_2$ respectively:
  \begin{align*}
    A_1& = A_1(f,w_1) = \int_{U}w_1 \cdot \nabla_{f}^2f + \frac{1}{2} (\nabla_ff)^2\left( \partial_xw_1 - \partial_z[fw_1]\right)\ud \mu \\
    A_2&= A_2(f, w_2) = \int_{U} -\nabla_f^2w_2 \cdot \nabla_{f}f + \frac{1}{2}  (\nabla_ff)^2\cdot \partial_z w_2 \ud \mu.
  \end{align*}
  If $\supp W\Subset U$, then $A_1=0$ and 
  \begin{equation}\label{eq:fvf domain smooth compact}
    A_2
     = \int_{U} \bigl(w_2\cdot \partial_zf + \nabla_fw_2\bigr)\cdot \nabla_f^2f \ud \mu 
     = \int_{U} w_2\cdot\bigl(2 \partial_zf \cdot \nabla_f^2 f - \nabla_f^3 f\bigr) \ud \mu.
   \end{equation}
\end{thm}
Note that $C$ only depends on the $L_\infty$ norm of $f$ and not $f$ itself. This will be important later, when we prove a version of Theorem~\ref{thm:fvf domain smooth} for intrinsic Lipschitz graphs. We cannot make $C$ completely independent of $f$ because the existence of $C$ is based on a compactness argument, and $\Gamma$ escapes to infinity when $f$ is large.  

Note also that in the compactly supported case, the first variation is independent of $w_1$. Indeed, vector fields parallel to $\nabla_f$ generate flows that preserve the foliation of $V_0$ by characteristic curves and thus preserve $\Gamma$. 

Before we prove Theorem~\ref{thm:fvf domain smooth}, we make some preliminary calculations.
\begin{lemma}\label{lem:prelim 1}
  With notation as in Theorem~\ref{thm:fvf domain smooth}, let $h\in \HH$ and let $u=\Pi(h)$.  Then
  \begin{equation}\label{eq:psi formula}
    \psi(h)=w_2(u)+(y(h)-f(u)) w_1(u)
  \end{equation}
  and $Y\psi(h)=w_1(u)$.
  When $h\in \Gamma$,
  \begin{equation}\label{eq:xpsi formula}
    X\psi(h) =\nabla_fw_2(u) - \nabla_ff(u)\cdot w_1(u)
  \end{equation}
\end{lemma}
\begin{proof}
  Note that $\Pi_*(X_p)=(1,0,-y(p))$. In particular, when $g\in \Gamma$, $\Pi_*(X_g)=(\nabla_f)_g$.

  Let $h\in \HH$ and $u=\Pi(h)$. Since $\phi_t$ sends cosets of $\langle Y\rangle$ to cosets of $\langle Y\rangle$, we have $\Pi(\phi_t(h))=\Pi(\phi_t(u))$ for all $h\in \HH$. Therefore,
  $$W(u)= \Pi_*(V_\psi(u)) = \partial_t[\Pi(\phi_t(u))]\big|_{t=0} = \partial_t[\Pi(\phi_t(h))]\big|_{t=0} = \Pi_*(V_\psi(h)).$$
  By \eqref{eq:contact field},
  \begin{align*}
    W(u) 
    &= \Pi_*(Y\psi(h) X_h) - \Pi_*(X \psi(h) Y_h) + \Pi_*(\psi(h) Z_h)\\
    &= ( Y\psi(h), 0, - Y\psi(h) y(h)+\psi(h)).
  \end{align*}
  In particular, $w_1(u)=Y\psi(h)$ and
  $$w_2(u) = \psi(h) + Y\psi(h) (f(u)-y(h))=\psi(h) + (f(u)-y(h)) w_1(u),$$
  which proves \eqref{eq:psi formula}.  
  
  Differentiating \eqref{eq:psi formula}, we find 
  \begin{equation}\label{eq:xpsi intermediate}
    X\psi(h)=\Pi_*(X_h)[w_2-fw_1] + y(h) \Pi_*(X_h)[w_1].
  \end{equation}
  When $h\in \Gamma$, we have $y(h)=f(u)$ and   $\Pi_*(X_h) = (\nabla_f)_u$, so 
  $$X\psi(h)=\nabla_f[w_2 - f w_1] + f \nabla_f w_1 = \nabla_fw_2 - \nabla_ff \cdot w_1,$$
  where all functions on the right are evaluated at $u$.
\end{proof}

\begin{lemma}\label{lem:domain to range} 
  With notation as in Theorem~\ref{thm:fvf domain smooth}, let $f_t\from U\to \R$ be such that $\Gamma_t=\Gamma_{f_t}$ and let $F(u,t)=f_t(u)$.  Let $u\in V_0$. If $f$ is smooth in a neighborhood of $u$, then 
  $$\left.\frac{\ud}{\ud t} f_t(\bar{\phi}_t(u))\right|_{t=0}=-\nabla_f w_2(u) + w_1(u) \cdot \nabla_f f(u)$$
  and
  $$\left.\frac{\ud}{\ud t} \nabla_{f_t} f_t(\bar{\phi}_t(u))\right|_{t=0}= -\nabla_f^2w_2(u) +w_1(u) \cdot \nabla_f^2f(u).$$
\end{lemma}
\begin{proof}
  Let $p=\Psi_f(u)$ so that $f(u)=y(p)$. Then $f_t(\bar{\phi}_t(u))=y(\phi_t(p))$, so 
  \begin{multline*}
    \left.\frac{\ud}{\ud t} f_t(\bar{\phi}_t(u))\right|_{t=0} = \left.\frac{\ud}{\ud t} y(\phi_t(p))\right|_{t=0} = y(V_\psi(p)) = - X\psi(p) \\
    = - \nabla_fw_2(u) + w_1(u) \nabla_ff(u),
  \end{multline*}
  as desired.
 
  Let $\gamma\from (-\epsilon,\epsilon)\to \Gamma$ be the horizontal curve in $\Gamma$ such that $\gamma(0)=p$ and $\gamma$ has unit $x$--speed (i.e., $x(\gamma(s))=x(p)+s$).  For any $t$, $\phi_t\circ \gamma$ parametrizes the horizontal curve in $\Gamma_t$ through $\phi_t(p)$, so
  $$\nabla_{f_t} f_t(\bar{\phi}_t(u)) = \frac{(y\circ \phi_t\circ \gamma)'(0)}{(x\circ \phi_t\circ \gamma)'(0)}.$$
  Let $T_p= \gamma'(0) = X_p + \nabla_ff(u) Y_p$.
  Commuting derivatives with respect to $s$ and $t$,
  $$\left.\frac{\ud}{\ud t} (x\circ \phi_t\circ \gamma)'(0) \right|_{t=0} = \partial_t \partial_s\left[x(\phi_t(\gamma(s)))\right](0,0) = \partial_s\left[x(V_\psi(\gamma(s)))\right](0) = T_p[x\circ V_\psi]$$
  and likewise $\frac{\ud}{\ud t} (y\circ \phi_t\circ \gamma)'(0) \big|_{t=0} = T_p[y\circ V_\psi]$.

  Furthermore, $\Pi_*(T_p) = \Pi_*(X_p) = (\nabla_f)_u$ and $x(V_\psi(p))=w_1(\Pi(p))$, so
  $$T_p[x\circ V_\psi] = T_p[w_1\circ \Pi] = \Pi_*(T_p)[w_1] = \nabla_f w_1(u).$$
  By \eqref{eq:contact field} and Lemma~\ref{lem:prelim 1}, $y\circ V_\psi= -X[\psi] = -\left(\nabla_fw_2 - \nabla_ff\cdot w_1\right)\circ \Pi$, so 
  \begin{multline*}
    T_p[y\circ V_\psi] = -\Pi_*(T_p)\left[\nabla_fw_2 - \nabla_ff\cdot w_1\right] = -\nabla_f\left[\nabla_f w_2 - \nabla_ff \cdot w_1\right](u) \\
    = \left(- \nabla^2_f w_2 + \nabla^2_f f \cdot w_1 + \nabla_ff \cdot \nabla_f w_1\right)(u)
  \end{multline*}
  By our choice of parametrization, $(x\circ \gamma)'(0)=1$ and $(y\circ \gamma)'(0)=\nabla_ff(u)$, so
  \begin{align*}
    \left.\frac{\ud}{\ud t} \nabla_{f_t} f_t(\bar{\phi}_t(u)) \right|_{t=0} 
    & = \frac{T_p[y\circ V_\psi] \cdot (x\circ \gamma)'(0) - (y\circ \gamma)'(0) \cdot T_p[x\circ V_\psi]}{(x\circ \gamma)'(0)^2}\\
    & = T_p[y\circ V_\psi] - \nabla_f f(u) \cdot T_p[x\circ V_\psi] \\
    & = -\nabla^2_f w_2(u) + \nabla^2_f f(u) \cdot w_1(u). 
  \end{align*}
\end{proof}

Theorem~\ref{thm:fvf domain smooth} follows.

\begin{proof}[{Proof of Theorem~\ref{thm:fvf domain smooth}}]
  Let
  $$F(t)=E_{\bar{\phi}_t(U)}(\Gamma_t)=\frac{1}{2} \int_{U} \left(\nabla_{f_t}f_t(\bar{\phi}_t(u))\right)^2 J_{\bar{\phi}_t}(u) \ud \mu(u)$$
  where $J_{\bar{\phi}_t}$ is the Jacobian determinant of $\bar{\phi}$. We have
  \begin{equation}\label{eq:Djt}
    \frac{\ud}{\ud t} J_{\bar{\phi}_t}\bigg|_{t=0} = \div_{V_0} W = \partial_x[x\circ W]+ \partial_z[z\circ  W]= \partial_xw_1 + \partial_z[w_2-fw_1].
  \end{equation}
  Exchanging derivative and integral and using Lemma~\ref{lem:domain to range},
  \begin{equation}
    \label{eq:fvf domain 1}
    F'(0) =\int_{U}\left(-\nabla_f^2w_2 + w_1 \cdot \nabla_{f}^2f\right)\nabla_{f}f + \frac{1}{2} (\nabla_ff)^2 \div_{V_0} W \ud \mu,
  \end{equation}
  Substituting \eqref{eq:Djt} into \eqref{eq:fvf domain 1} gives us $F' =A_1+A_2.$

  \HERE
  
  In order to prove \eqref{eq:fvf domain full}, it suffices to show that there are $\epsilon_0=\epsilon_0(U, \psi, \|f\|_{L_\infty(U)})>0$ and $C=C(U, \psi, \|f\|_{L_\infty(U)})>0$ such that $|F''(t)|\le C(E_{U}(\Gamma) +\mu(U))$ for all $t\in [-\epsilon_0,\epsilon_0]$.  Let $K_u(t)=J_{\bar{\phi}_t}(u)$,   $N_u(t)=\nabla_{f_t}f_t(\bar{\phi}_t(u))$.  Then
  \begin{equation}\label{eq:F''}
    F''=\frac{1}{2} \int_{U} \left(N_u^2 K_u\right)'' \ud \mu= \int_{U} \left((N_u')^2 +N_uN_u''\right)K_u+ 2N_uN_u'K_u'+\frac{N_u^2K_u''}{2} \ud \mu.
  \end{equation}
  By continuity, there is a $c_1(\psi)>0$ such that $\max\{|K_u(t)|, |K_u'(t)|, |K_u''(t)|\} \le c_1(\psi)$ for all $u\in U$ and $t\in [-\epsilon,\epsilon]$.

  Let $M=\{vY^r\in \HH\mid v\in U, |r|\le \|f\|_{L_\infty(U)}\}$. Since $\phi_0=\id_\HH$, $X[x\circ \phi_0]=1$; by continuity, we may suppose that $\epsilon_0$ is sufficiently small (depending on $\psi$, $U$, and $\|f\|_{L_\infty(U)}$) that $X[x\circ \phi_t](q)>\frac{1}{2}$ for all $q\in M$.  Choose $c_2$ such that
  $$c_2 \ge \big|\partial_t^i V[v\circ \phi_t](q)\big|$$
  for all $q\in M$, $i=0,1,2$, $V=X,Y$, and $v=x,y$. 

  Let $u\in U$ and let $p=\Psi_f(u)$.  Let $T$ be the left-invariant vector field $T = X + \nabla_ff(u) Y$.  As in the proof of Lemma~\ref{lem:domain to range}, we have
  $$N_u(t) = \frac{T[y\circ \phi_t](p)}{T[x\circ \phi_t](p)} =: \frac{\omega(t)}{\chi(t)}.$$

  Since $\phi_t$ sends cosets of $\langle Y \rangle$ to cosets of $\langle Y \rangle$, we have $Y[x\circ \phi_t]=0$. Therefore, $\chi(t)= X[x\circ \phi_t](p) >\frac{1}{2}$. By our choice of $c_2$, for $i=0,1,2$ and $|t|\le \epsilon_0$,
  $$\big|\omega^{(i)}(t)\big| = \left|\partial_t^iT[y\circ \phi_t](p)\right| \le  c_2(1+\nabla_f f(u)),$$
  $$\big|\omega^{(i)}(t)\big| = \left|\partial_t^iX[y\circ \phi_t](p)+\nabla_ff(u) \cdot \partial_t^iY[y\circ \phi_t](p)\right| \le  c_2(1+\nabla_f f(u)),$$
  and
  $$\big|\chi^{(i)}(t)\big| =\big|\partial_t^iX[y\circ \phi_t](p)\big| \le c_2.$$
  Thus for $|t|\le \epsilon_0$,
  \begin{equation*}
    \left|N_u'(t)\right|
    =\left|\frac{\omega'(t) \chi(t) - \omega(t) \chi'(t) }{\chi(t)^2}\right|\lesssim c_2(1+\nabla_f f(u)),
  \end{equation*}
  and
  \begin{equation*}
    \left|N_u''(t)\right|
    = \left|\omega(t)\biggl(\frac{2\chi'(t)^2}{\chi(t)^3} - \frac{\chi''(t)^2}{\chi(t)^2}\biggr) - 2 \omega'(t)\frac{\chi'(t)}{\chi(t)^2} + \frac{\omega''(t)}{\chi(t)}\right|
    \lesssim c_2(1+\nabla_ff(u)).
  \end{equation*}
  We apply these bounds to \eqref{eq:F''} to get
  $$\left|F''(t)\right|\lesssim \int_U c_1 c_2^2 (1+\nabla_f f(u))^2 \ud \mu(u) \le C(\mu(U) + E_U(\Gamma))$$
  for some $C=C(\psi, U, \|f\|_{L_\infty(U)})$. Equation \eqref{eq:fvf domain full} then follows from Taylor's theorem.
  
  It remains to consider the case that $\supp W\Subset U$.  Let $K\subset U$ be a closed set with piecewise-smooth boundary that contains $\supp W$.
  We calculate
  \begin{align}
\notag    A_1 
    &= \frac{1}{2} \int_{K} \nabla_f\left[(\nabla_ff)^2\right]\cdot w_1 + (\nabla_ff)^2\left( \partial_xw_1 - f \partial_zw_1 - w_1 \partial_zf\right)\ud \mu\\
    \label{eq:fvf domain intermediate}
    &= \frac{1}{2} \int_{K} \nabla_f\left[(\nabla_ff)^2\right]\cdot w_1 + (\nabla_ff)^2 \cdot \nabla_f  w_1 - (\nabla_ff)^2\cdot w_1 \cdot \partial_zf\ud \mu.
  \end{align}
  Corollary~\ref{cor:integration by parts} with $g=(\nabla_ff)^2$ and $h=w_1$ implies that $A_1=0$.

  We decompose $A_2$ as
  $$A_2  =\int_{K} -\nabla_f^2w_2 \cdot \nabla_ff \ud \mu + \int_{K} \frac{1}{2}  (\nabla_ff)^2\cdot \partial_z w_2 \ud \mu =: Q_1 + Q_2.$$
  By Corollary~\ref{cor:integration by parts} with $g=\nabla_{f}w_2$ and $h=\nabla_ff$, 
  $$Q_1 = - \int_{K} \nabla_f w_2 \cdot \nabla_ff \cdot \partial_zf\ud\mu + \int_{K} \nabla_f^2 f \cdot \nabla_fw_2 \ud \mu := Q_3 + Q_4.$$
  Applying Corollary~\ref{cor:integration by parts} again with $g=w_2$, $h=\nabla_ff \cdot \partial_zf$, we find
  \begin{multline*}
    Q_3 = \int_{K}- w_2 \cdot \nabla_ff \cdot (\partial_zf)^2 + w_2 \cdot \nabla_f\left[\nabla_ff \cdot \partial_zf\right] \ud \mu\\
    = \int_{K} w_2 \cdot \nabla_ff \cdot \left(-(\partial_zf)^2 + \nabla_f[\partial_zf]\right)+w_2\cdot \partial_zf \cdot \nabla_f^2f \ud \mu.
  \end{multline*}
  Furthermore, 
  $$[\nabla_f , \partial_z]=-[\partial_z, \partial_x-f\partial_z]=\partial_z f\cdot \partial_z,$$
  so $-(\partial_zf)^2 + \nabla_f[\partial_zf]= \partial_z[\nabla_ff]$.  Thus
  \begin{equation}
    Q_3
    = \int_{K} w_2 \cdot \nabla_ff \cdot \partial_z [\nabla_f f]+ w_2\cdot \partial_zf \cdot \nabla_f^2f \ud \mu.
  \end{equation}

  Integrating $Q_2$ by parts gives
  $$Q_2=\int_{K} \frac{1}{2}  (\nabla_ff)^2\cdot \partial_z w_2 \ud \mu = - \int_{K} w_2\cdot \nabla_ff \cdot \partial_z[\nabla_ff] \ud \mu,$$
  so
  $$A_2  = Q_3+Q_4+Q_2 = \int_{K^+} \left(w_2\cdot \partial_zf + \nabla_fw_2\right)\cdot \nabla_f^2f \ud \mu.$$
  This is the first part of \eqref{eq:fvf domain smooth compact}. Applying 
  Corollary~\ref{cor:integration by parts} once more with $g=w_2$ and $h=\nabla^2_ff$, we get
  $$A_2 = \int_{U} w_2\cdot\bigl(2 \partial_zf \cdot \nabla_f^2 f - \nabla_f^3 f\bigr) \ud \mu.$$
\end{proof}

We will prove a first variation formula for intrinsic Lipschitz graphs by approximating them by smooth graphs and applying Theorem~\ref{thm:fvf domain smooth}. Note, however, that when $f$ is intrinsic Lipschitz, $\nabla_f w$ is generally not smooth, so $\nabla_f^2w$ may be undefined. We thus introduce a new operator. For any intrinsic Lipschitz function $f\from V_0\to \R$, any smooth $w\from V_0\to \R$, and any $p\in V_0$, let $\lambda_p(t)=\Psi_f(p)(X+\nabla_ff Y)^t$ be the intrinsic tangent line to $\Gamma_f$ at $\Psi_f(p)$ and let 
\begin{equation}\label{eq:def Delta}
  \Delta_f w(p) = (w\circ \Pi\circ \lambda_p)''(0).
\end{equation}
This is defined almost everywhere on $V_0$. When $f$ is smooth, the first two derivatives of the characteristic curve through $p$ agree with the first two derivatives of $\lambda_p$, so $\Delta_f w=\nabla_f^2w$. In general, if $p=(x_0,0,z_0)$, then 
$$\Pi(\lambda_p(t)) = \left(x_0+t, 0, z_0- f(p) t -\frac{t^2}{2} \nabla_ff(p)\right),$$
so
\begin{align}
\notag  \Delta_f w(p)
  &= \left(\partial_x -  f(p)\partial_z\right)^2[w](p) - \nabla_f f(p)  \cdot \partial_z w(p) \\
\notag  & = \nabla_f\left[\partial_x w - f(p) \partial_z w\right](p) - \nabla_f f(p)\cdot \partial_z w(p)\\
  \label{eq:delta f expand}
  & = \nabla_f[\partial_x w](p) - f(p)\nabla_f[\partial_z w](p) - \nabla_f f(p)\cdot \partial_z w(p).
\end{align}

\begin{thm}[First variation formula for intrinsic Lipschitz graphs]\label{thm:fvf domain lip}
  Let $\Gamma=\Gamma_f$ be the intrinsic graph of an intrinsic Lipschitz function $f\from V_0\to \R$. Let $\phi_t$, $\psi$, $\bar{\phi}_t$, and $W$ be as in Theorem~\ref{thm:fvf domain smooth}. For each $t\in [-\epsilon,\epsilon]$, let $f_t$ be the function such that $\Gamma_{f_t}=\phi_t(\Gamma)$.

  Let $U\subset V_0$ be a bounded open subset and suppose that $\supp W\Subset U$. For any $w\in C^\infty(U)$ and any intrinsic Lipschitz function $g$, let
  \begin{align*}
    B_2(g, w) & = -\Delta_g w \cdot \nabla_{g}g + \frac{1}{2}  (\nabla_gg)^2\cdot \partial_z w,\\
    B_1(g, w) & = g B_2(g, w) - \frac{3}{2} (\nabla_gg)^2 \cdot \nabla_{g} w.
  \end{align*}

  Then there is a $C=C(\psi, U, \|f\|_{L_\infty(U)})$ such that 
  \begin{equation}\label{eq:fvf domain lip}
    \left|E_{U}(f_t) - E_{U}(f) - t \int_U  B_1(f,x\circ W) + B_2(f, z \circ W) \ud\mu\right|\le C \left(E_{U}(f) +\mu(U)\right) t^2
  \end{equation}
  for all $|t|\le \epsilon$.
\end{thm}

The function $B_2$ above is simply the integrand in the definition of $A_2$, with $\nabla^2_fw$ replaced by $\Delta_fw$. The following lemma shows that when $f$ is smooth, Theorem~\ref{thm:fvf domain lip} follows from Theorem~\ref{thm:fvf domain smooth}.
\begin{lemma}\label{lem:B1 and B2}
  Let $U\subset V_0$ be a bounded open subset. Let $B_1$ and $B_2$ be as in Theorem~\ref{thm:fvf domain lip}.  
  When $f$ is smooth and $w\in C^\infty_c(U)$,
  \begin{equation}\label{eq:B1 eq B2}
    \int_U B_1(f, w) \ud \mu = \int_U B_2(f, fw)\ud \mu,
  \end{equation}
  and $\int_U B_2(f,w)\ud \mu =A_2(f,w)$, where $A_2$ is as in Theorem~\ref{thm:fvf domain smooth}. Consequently, if $W$ is a smooth vector field with $\supp W\Subset U$ and if $w_2=(x\circ W)\cdot f + z\circ W$, then
  $$A_2(f,w_2)=\int_U B_2(f, w_2)\ud \mu = \int_U B_1(f, x\circ W) + B_2(f,z\circ W)\ud \mu.$$
\end{lemma}
\begin{proof}
  Since $f$ is smooth, we have $\Delta_f w=\nabla_f^2w$ and $\int_U B_2(f,w)\ud \mu =A_2(f,w)$.
  Let $\lambda=\nabla_{f}f$. Then
  $$\Delta_f[fw]-f \Delta_f[w] = 2 \nabla_f f\cdot \nabla_f w + \nabla^2_{f} f\cdot w = 2 \lambda \cdot \nabla_f w + \nabla_{f} \lambda \cdot w,$$
  so
  $$B_2(f, f w) - fB_2(f, w) = - 2 \lambda^2 \cdot \nabla_{f} w - \lambda \cdot \nabla_{f}\lambda \cdot w + \frac{1}{2} \lambda^2 \cdot \partial_z f \cdot w.$$

  By Corollary~\ref{cor:integration by parts}, with $f=w$ and $h=\frac{1}{2} \lambda^2$,
  \begin{align*}
    \int_U B_2(f, f w)\ud\mu
    &= \int_U f B_2(f, w) - 2 \lambda^2 \cdot \nabla_{f} w - \lambda \nabla_{f}\lambda \cdot w + \frac{1}{2} \lambda^2 \cdot \partial_z f \cdot w\ud\mu \\
    &= \int_U f B_2(f, w) - \frac{3}{2} \lambda^2 \cdot \nabla_{f} w \ud\mu\\
    &= \int_U B_1(f, w)\ud\mu,
  \end{align*}
  as desired.
\end{proof}

\begin{proof}[Proof of Theorem~\ref{thm:fvf domain lip}]
  Let $f^k$ be a sequence of approximating smooth functions for $f$ satisfying Theorem~\ref{thm:CMPSC approx}, so that $f^k\to f$ uniformly, there is a $c>0$ such that $\|\nabla_{f^k}f^k\|_\infty< c$ for all $k$, and $\nabla_{f^k}f^k\to \nabla_ff$ pointwise almost everywhere. For $w\in C^\infty(V_0)$, we have
  $$\lim_{k\to \infty} \nabla_{f^k}w = \lim_{k\to \infty} \partial_x w -f^k\partial_z  = \nabla_fw$$
  uniformly and 
  $$\lim_{k\to \infty}\Delta_{f^k} w \stackrel{\eqref{eq:delta f expand}}{=} \lim_{k\to \infty} \nabla_{f^k}[\partial_x w] - \nabla_{f^k} f^k\cdot \partial_z w - f^k\nabla_{f^k}[\partial_z w]= \Delta_fw$$
  pointwise almost everywhere. In particular, if $w\in C^\infty_c(V_0)$, then $B_1(f^k,w)$ and $B_2(f^k,w)$ are bounded by a function of $c$ and $w$ and converge  pointwise a.e.\ to $B_1(f,w)$ and $B_2(f^k,w)$, respectively.

  For each $t\in [-\epsilon,\epsilon]$, the image $\phi_t(\Gamma_{f^k})$ is a smooth intrinsic graph; let $f^k_t$ be such that $\Gamma_{f^k_t}=\phi_t(\Gamma_{f^k})$.
  Let $K$ be a bounded set containing $\Psi_{f^k}(U)$ for every $k$. 
  For all $u\in U$, we have
  $$\left|f^k_t(u)-f_t(u)\right|\le \left|f^k\left(\bar{\phi}_t^{-1}(u)\right)-f\left(\bar{\phi}_t^{-1}(u)\right) \right| \Lip\left(\phi_t|_K\right).$$
  Since $f^k\to f$ uniformly, this implies $f^k_t\to f_t$ uniformly.

  Likewise, for any $t\in [-\epsilon,\epsilon]$, if $\lim_{k\to \infty} \nabla_{f^k}f^k(u)= \nabla_{f}f(u)$, then
  $$\lim_{k\to \infty} \nabla_{f^k_t}f^k_t(\bar{\phi}_t(u))=\nabla_{f_t}f_t(\bar{\phi}_t(u)),$$
  so $\nabla_{f^k_t}f^k_t\to \nabla_{f_t}f_t$ pointwise a.e.
  Thus, by dominated convergence,
  \begin{equation}\label{eq:fvf lip energy limit}
    \lim_{k\to \infty} E_U(f^k_t) = E_U(f_t).
  \end{equation}

  Theorem~\ref{thm:fvf domain smooth} and Lemma~\ref{lem:B1 and B2} imply that there are $\epsilon_0,C>0$ such that for any $k$ and any $t\in [-\epsilon_0, \epsilon_0]$, 
  \begin{align}\label{eq:fvf lip recap}
\notag    \Bigl|E_U(f^k_t)
    &- E_{U}(f^k) - t A_2\big(f^k, (x\circ W)\cdot f + z\circ W\big)\Bigr| \\
\notag    & = \left|E_U(f^k_t)- E_{U}(f^k) - t\int_U B_1\big(f^k, x\circ W\big) + B_2\big(f^k, z\circ W\big)\ud \mu\right| \\
    & \le C \bigl(E_{U}(f^k) +\mu(U)\bigr) t^2.
  \end{align}
  Taking the limit as $k\to \infty$ and using dominated convergence to exchange the integral and the limit, we get
  \begin{equation*}
    \left|E_U(f_t) - E_{U}(f) - t\int_U B_1(f, x\circ W) + B_2(f, z\circ W)\ud \mu \right|\le C \left(E_{U}(\Gamma) +\mu(U)\right) t^2,
  \end{equation*}
  as desired.
\end{proof}
This gives a two-part condition for contact harmonicity for intrinsic Lipschitz functions. Namely, an intrinsic Lipschitz $f$ is contact harmonic on $U$ if and only if
$$\int_U B_1(f, w)\ud \mu=\int_U B_2(f, w)\ud \mu=0$$
for every $w\in C^\infty_c(U)$. In contrast, when $f$ is smooth and $\int_U B_2(f, w)\ud \mu=0$ for all $w\in C^\infty_c(U)$, \eqref{eq:B1 eq B2} implies that
$$\int_U B_1(f, w)\ud\mu = \int_U B_2(f, fw)\ud \mu = 0$$
for all $w\in C^\infty_c(U)$, so the condition on $B_2$ suffices to characterize contact harmonicity. We do not know whether the condition on $B_2$ suffices to characterize contact harmonicity when $f$ is not smooth.



\subsection{Vertical first variation for graphs with herringbone singularities}\label{sec:fvf herring}
In this section, we will refine the formulas in Theorem~\ref{thm:fvf domain smooth} for compactly supported vertical contact variations on a class of intrinsic Lipschitz graphs that are singular along a horizontal curve. We show that if $\Gamma$ is contact harmonic, the horizontal curves near the singularity must satisfy the equal-slope condition \eqref{eq:tau conserv}. 

We first define the class of singularities we are interested in.
\begin{defn}\label{def:smooth herringbone}
  An intrinsic Lipschitz graph with a \emph{smooth herringbone singularity} consists of two smooth intrinsic graphs $\Gamma^+$ and $\Gamma^-$ meeting along a horizontal curve $C$ and satisfying the properties below.
  
  Let $U\subset V_0$ be an open set and let $\Gamma=\Gamma_f\subset \HH$ be an intrinsic Lipschitz graph over $U$.  Suppose that the characteristic nexus of $\Gamma$ is a smooth horizontal curve $C\subset \Gamma$ and let $\gamma\from I\to V_0$, $\gamma(t)=(t,0,\gamma_z(t))$ parametrize $\Pi(C)$.  Suppose that $\gamma$ cuts $U$ into two connected components, $U^+=\{(x,0,z)\in U\mid z>\gamma_z(x)\}$ and $U^-=\{(x,0,z)\in U\mid z<\gamma_z(x)\}$. Let $\Gamma^\pm=\Psi_f(U^\pm)$.

  Suppose that $f$ is smooth on $U^+$ and $U^-$ (but generally not on $\gamma$), so that $\Gamma^+$ and $\Gamma^-$ are foliated by horizontal curves. We require that the foliations extend to the boundary in the following sense:
  \begin{itemize}
  \item
    There is a neighborhood $N$ of $C$ such that the projection $\pi\from \HH\to \R^2$, $\pi(x,y,z)=(x,y)$ restricts to an embedding of $N\cap \Gamma$.
  \item
    Let $H^\pm$ be the horizontal foliation of $\Gamma^\pm$.  Let $M=\pi(N\cap \Gamma)$ and let $M^\pm=\pi(N\cap \Gamma^\pm)$. The projection $\pi_*(H^+)$ can be extended to a smooth foliation $F^+$ defined on a neighborhood of $\overline{M^+}$. Likewise, $\pi_*(H^-)$ extends to a smooth foliation $F^-$ defined on a neighborhood of $\overline{M^-}$. These foliations are transverse to $\pi(C)$ and their tangent lines have bounded slopes.
  \end{itemize}
  Then $C$ is a smooth herringbone singularity of $\Gamma$.
\end{defn}

The behavior of $\Gamma$ near $C$ is governed by the slopes of $\pi(C)$, $F^+$, and $F^-$.  Let $c=(t,c_y,c_z) \from I\to \HH$ be a parametrization of $C$. Let $\sigma^0(t)=c_y'(t)$ be the slope of $\pi\circ c$. Let $\sigma^+(t)$ (resp.\ $\sigma^-(t)$) be the slope of $F^+$ (resp.\ $F^-$) at $\pi(c(t))$. We will see in Lemma~\ref{lem:near sing} that $\sigma^+(t)<\sigma^0(t)<\sigma^-(t)$ for all $t\in I$.

Smooth herringbone singularities are either \emph{left-pointing} or \emph{right-pointing}. Since $\pi$ is a homeomorphism from $N\cap \Gamma$ to $M$, the sets $M^+$ and $M^-$ are separated by the projection $\pi(C)$. If $M^+$ is above $\pi(C)$ (i.e., $M^+=\{(x,y)\in \pi(N)\mid y>c_y(x)\}$), we say $C$ is a right-pointing singularity (because $\pi_*(H^\pm)$ looks like $\ggg$). If $M^+$ is below $\pi(C)$, we say $C$ is a left-pointing singularity (because $\pi_*(H^\pm)$ looks like $\lll$).

We now state the first variation formula for vertical contact variations, i.e., contact variations where the corresponding field $W$ on $V_0$ is vertical.
\begin{thm}[First variation formula with herringbone singularities]\label{thm:fvf herringbone}
  Let $U\subset V_0$ be an open subset, let $f\from U\to \R$, and let $\Gamma=\Gamma_f$ be an intrinsic Lipschitz graph with a smooth herringbone singularity $C$. Let $\sigma^0$, $\sigma^+$, and $\sigma^-$ be as above, and let $\gamma\from I\to V_0$, $\gamma(t)=(t,0,\gamma_z(t))$ be a parametrization of $\Pi(C)$.

  Let $w_2\in C^\infty_c(U)$, let $\psi=w_2\circ \Pi$ be a potential, and let $V_\psi= - (X\psi)Y + w_2Z$ be the corresponding field. We can write $V_\psi=\Pi^*(W)$, where $W=w_2Z$ is a vector field on $V_0$. Let $\phi_t\from \HH\to \HH$, $t\in (-\epsilon,\epsilon)$ be the flow of $V_\psi$, and let $\bar{\phi}_t\from V_0\to V_0$ be the flow of $W$.

  Let $f_t$ be the function such that $\Gamma_{f_t}=\phi_t(\Gamma)$.
  There are $\epsilon_0=\epsilon_0 (U, w_2, \|f\|_{L_\infty(U)})>0$ and $D=D(U, w_2, \|f\|_{L_\infty(U)})>0$ such that 
  \begin{equation}\label{eq:fvf herringbone full}
    \left|E_{U}(f_t) - E_{U}(f) - A_2 t\right|\le D \left(E_{U}(f) +\mu(U)\right) t^2
  \end{equation}
  for all $t\in [-\epsilon_0,\epsilon_0]$, where
  \begin{align*}
    A_2
    &= \int_{U} -\nabla_f^2w_2 \cdot \nabla_{f}f +\frac{1}{2}  (\nabla_ff)^2\cdot \partial_z w_2 \ud \mu\\
    & = \frac{1}{2} \int_{I} w_2(\gamma(s)) \cdot \delta(s) \ud s + \int_{U} \bigl(w_2\cdot \partial_zf + \nabla_fw_2\bigr)\cdot \nabla_f^2f \ud \mu,
  \end{align*}
  and $\delta(s)=(\sigma^+(s)-\sigma^0(s))^2 - (\sigma^0(s)-\sigma^-(s))^2$.

  In particular, if $\Gamma$ is contact harmonic on $U$, then $\delta(s)=0$ and thus $\sigma^0(s)=\frac{\sigma^+(s)+\sigma^-(s)}{2}$ for all $s\in I$. That is, the slope of $\pi(C)$ is the average of the slopes of $F^+$ and $F^-$.
\end{thm}
The condition on the slope of $\pi(C)$ is analogous to the condition in \cite{ChengHwangYang} that a singular curve in an $H$--minimal $Z$--graph bisects the foliation lines on either side. 



Before we prove Theorem~\ref{thm:fvf herringbone}, we examine the behavior of $\Gamma$ near $C$.  
\begin{lemma}\label{lem:near sing}
  With notation as in Theorem~\ref{thm:fvf herringbone}, for all $s\in I$, $\sigma^+(s)<\sigma^0(s)<\sigma^-(s)$.  Let $\delta=1$ if $C$ is right-pointing and $\delta=-1$ if $C$ is left-pointing. Let $J\subset I$ be a compact interval. Then for all $s\in J$ and all sufficiently small $\nu>0$,
  \begin{equation}\label{eq:singularity f plus}
    f(\gamma(s)Z^{\nu})-f(\gamma(s)) = 
    \delta \sqrt{2 \left(\sigma^0(s)- \sigma^+(s)\right) \nu} + O(\nu)
  \end{equation}
  \begin{equation}\label{eq:singularity f minus}
    f(\gamma(s)Z^{-\nu})-f(\gamma(s)) = 
    -\delta \sqrt{2 \left(\sigma^-(s)- \sigma^0(s)\right) \nu} + O(\nu)
  \end{equation}
  \begin{equation}\label{eq:partial f plus}
    \partial_z f(\gamma(s)Z^{\nu}) = 
    \delta \sqrt{\frac{\sigma^0(s) - \sigma^+(s)}{2 \nu}} + O(1)
  \end{equation}
  \begin{equation}\label{eq:partial f minus}
    \partial_z f(\gamma(s)Z^{-\nu}) = 
    \delta \sqrt{\frac{\sigma^-(s) - \sigma^0(s)}{2 \nu}} + O(1),
  \end{equation}
  where the implicit constants depend on $f$ and $J$. In particular, $\partial_zf$ is $L_1$ on a neighborhood of $\gamma(J)$.
\end{lemma}  
\begin{proof}
  Let $J'\Subset I$ be an interval such that $J\Subset J'$. Let $\epsilon>0$ be sufficiently small that for any $s\in J'$, there are unique unit $x$--speed horizontal curves $\lambda^\pm_s\from (-\epsilon,\epsilon)\to \HH$ such that $\lambda^\pm_s(0)=c(s)$ and the projection $\pi\circ \lambda^+_s$ (resp.\ $\pi\circ \lambda^-_s$) is a leaf of $F^+$ (resp.\ $F^-$), and let $\Lambda^\pm(s,t)=\lambda^\pm_s(t)$.
  Then $\Lambda^+$ is a smooth map defined on $D=J' \times(-\epsilon,\epsilon)$, and its image contains a neighborhood of $c(J)$ in $\Gamma^+$. 
  Let $\zeta=z\circ \Pi \circ \Lambda^+$ and let $s\in J$. Since $\gamma(s)=\Pi(c(s))=\Pi(\Lambda^+(s,0))$, we have $\zeta(s,0)=\gamma_z(s)$.

  We expand $\zeta$ around $(s,0)$.
  Since $s\mapsto \Lambda^+(s,0)=c(s)$ is a smooth, unit $x$--speed horizontal curve, \eqref{eq:horiz proj} implies $\partial_s\zeta(s,0)=-y(c(s))$ and
  $$\partial^2_{s}\zeta(s,0)=-(y\circ c)'(s)=-\sigma^0(s).$$
  Likewise, for any $s$, the map $t\mapsto \Lambda^+(s,t) = \lambda^\pm_s(t)$ is a smooth, unit $x$--speed horizontal curve, so 
  $\partial_t\zeta(s,t)=-y(\lambda_s^+(t))$, and
  $$\partial^2_t\zeta(s,t)=-(y\circ \lambda^+_s)'(t).$$
  Setting $t=0$, we get $\partial_t\zeta(s,0)=-y(c(s))$ and $\partial^2_t\zeta(s,0)=-\sigma^+(s)$, so $\partial_s \partial_t\zeta(s,0)=-(y\circ c)'(s)=-\sigma^0(s)$. Thus
  \begin{align}
\notag    \zeta(s-t,t) 
    &= \gamma_z(s) + y(c(s)) t - y(c(s)) t - \frac{\sigma^0(s)}{2} t^2 + \sigma^0(s) t^2 - \frac{\sigma^+(s)}{2} + O(t^3) \\
    \label{eq:zeta expand}
    &= \gamma_z(s) + \frac{\sigma^0(s) - \sigma^+(s)}{2} t^2 + O(t^3).
  \end{align}

  For every $s\in J'$ and $0<t<\epsilon$, either $\Lambda^+(s-t,t)$ or $\Lambda^+(s+t,-t)$ lies in $\Gamma^+$, so $\zeta(s-t,t) > \gamma_z(s)$ or $\zeta(s+t,-t) > \gamma_z(s)$. In either case, \eqref{eq:zeta expand} implies $\sigma^+(s)<\sigma^0(s)$. By a similar argument, $\sigma^0(s)<\sigma^-(s)$ for all $s\in J'$. Thus, if $C$ is right-pointing, then $M^+=\pi(N\cap \Gamma^+)$ is above $\pi(C)$, and the horizontal foliations of $\Gamma$ project to lines of the form $\ggg$, while if $C$ is left-pointing, they project to lines of the form $\lll$.

  For the rest of this proof, we suppose that $C$ is right-pointing. If $C$ is a left-pointing singularity, we can rotate $\Gamma$ by 180$^\circ$ around the $z$--axis to produce a new graph $\widehat{\Gamma}=s_{-1,-1,1}(\Gamma)$ with a right-pointing singularity. The slopes $\sigma^0$, $\sigma^+$, and $\sigma^-$ stay the same under this symmetry, but the sign of $f$ flips, so the left-pointing case follows from the right-pointing case. 

  Let $s\in J$. The smoothness of $\zeta$ and \eqref{eq:zeta expand} imply that for any sufficiently small $\nu>0$, there is a unique $t>0$ such that $\zeta(s+t,-t)-\gamma_z(s)=\nu$. Let $p=\Lambda^+(s+t,-t)\in \Gamma$. Since $C$ is right-pointing, $p\in \Gamma^+$. Then
  $$\Pi(p)=(s+t-t, \zeta(s+t,-t))=(s, \gamma_z(s) + \nu)=\gamma(s)Z^\nu,$$
  so
  \begin{equation}\label{eq:f and lambda}
    f\big(\gamma(s) Z^{\nu}\big)=y(p)=y(\Lambda^+(s+t,-t)).
  \end{equation}

  We thus consider the relationship between $\nu$ and $t$.  By \eqref{eq:zeta expand},
  \begin{equation}\label{eq:nu expand}
    \nu =\frac{\sigma^0(s) - \sigma^+(s)}{2} t^2 + O(t^3).
  \end{equation}
  For $s\in J$, $\sigma^0(s) - \sigma^+(s)$ is bounded away from zero and bounded by a function of the intrinsic Lipschitz constant of $f$, so $\nu\approx t^2$ and 
  $$t= \sqrt{\frac{2\nu}{\sigma^0(s) - \sigma^+(s)}+O(t^3)} = \sqrt{\frac{2\nu}{\sigma^0(s) - \sigma^+(s)}} + O\left(\frac{t^3}{\sqrt{\nu}}\right) = \sqrt{\frac{2\nu}{\sigma^0(s) - \sigma^+(s)}} + O(\nu).$$
  Since $\partial_s[y\circ \Lambda^+](s,0) = \sigma^0(s)$ and $\partial_t[y\circ \Lambda^+](s,0) = \sigma^+(s)$, 
  \begin{multline*}
    f(\gamma(s) Z^{\nu}) 
    = y(\Lambda^+(s+t,-t)) 
    = y(\Lambda^+(s,0)) + (\sigma^0(s)-\sigma^+(s))t + O(t^2)\\
    =f(\gamma(s)) + \sqrt{2\nu(\sigma^0(s)-\sigma^+(s))} + O(\nu),
  \end{multline*}
  with implicit constants depending on $f$ and $J$.
  This proves \eqref{eq:singularity f plus}.

  Differentiating \eqref{eq:nu expand} gives
  $$\frac{\ud \nu}{\ud t} = (\sigma^0(s) - \sigma^+(s)) t + O(t^2).$$
  Therefore, 
  \begin{align*}
    \frac{\ud}{\ud \nu} f(\gamma(s) Z^{\nu})
    &\stackrel{\eqref{eq:f and lambda}}{=} \frac{\ud t}{\ud \nu} \frac{\ud}{\ud t} y(\Lambda^+(s+t,-t))\\
    &= \left(\frac{1}{(\sigma^0(s) - \sigma^+(s)) t} + \frac{O(t^2)}{(\sigma^0(s) - \sigma^+(s))^2 t^2}\right)\left(\sigma^0(s)-\sigma^+(s)+O(t)\right)\\
    &= t^{-1} + O(1)\\
    &= \sqrt{\frac{\sigma^0(s) - \sigma^+(s)}{2\nu}} + O(1),
  \end{align*}
  proving \eqref{eq:partial f plus}. The argument for the case $\nu<0$ and the equations \eqref{eq:singularity f plus} and \eqref{eq:partial f plus} is symmetric.
\end{proof}

Now we prove Theorem~\ref{thm:fvf herringbone}.
\begin{proof}[{Proof of Theorem~\ref{thm:fvf herringbone}}]
  By Lemma~\ref{lem:near sing}, $f$ is smooth on $U^\pm=\Pi(\Gamma^\pm)$, but not on the curve $\gamma=\Pi(C)$ dividing $U^+$ and $U^-$. Indeed, $\partial_zf\to \infty$ near $\gamma$. Nevertheless, $\nabla_ff$ is equal to the slopes of the curves in $F^+$ and $F^-$, so $\nabla_ff$ is bounded but discontinuous near $\gamma$. Let $K\subset U$ be a closed set with piecewise smooth boundary such that $\supp w_2 \Subset K$.
  Let $K^\pm=U^\pm\cap K$. We define $\lambda^+\from \overline{K^+} \to \R$,
  $$\lambda^+(p)=\begin{cases} \nabla_f f(p) & p\in K^+\\
    \sigma^+(x(p)) & p\in \gamma
  \end{cases}$$
  and define $\lambda^-\from \overline{K^-} \to \R$ likewise.

  These maps are continuous away from $\gamma$. To show continuity on $\gamma$, let $N$ be a neighborhood of $C$ as in Definition~\ref{def:smooth herringbone} and let $s^+\from \pi(N\cap \Gamma) \to \R$ so that $s^+(b)$ is the slope of $F^+$ at $b$. Then for any $v=(x_v,z_v)\in \Pi(N\cap \Gamma^+)$, $\nabla_ff(v)$ is the slope of the horizontal curve through $\Psi_f(v)$, i.e.,
  \begin{equation}\label{eq:nablaff formula}
    \nabla_ff(v)=s^+(\pi(\Psi_f(v)))=s^+(x_v,f(v)).
  \end{equation}
  This is continuous in a neighborhood of $\gamma$, so $\lambda^+$ and $\lambda^-$ are continuous.

  Furthermore, we can bound $\nabla_f \lambda^\pm$ and $\partial_z \lambda^\pm$ near $\gamma$. The horizontal derivative $\nabla_f\lambda^\pm$ is the derivative of $s^\pm$ along a leaf of $F^\pm$, so $\nabla_f\lambda^\pm\in L_\infty(K^\pm)$. 
  By \eqref{eq:nablaff formula} and Lemma~\ref{lem:near sing},
  \begin{align}
\notag    \partial_z \lambda^\pm(x,0,z)
    &=\partial_z f(x,0, z) \cdot \partial_y s^+(x,f(x,0,z))\\
\label{eq:pz lambda l1}    &=\left(\pm \sqrt{\frac{\sigma^0(s) - \sigma^+(s)}{2 (z-\gamma_z(x))}} + O(1)\right) \cdot \partial_y s^+(x,f(x,z)).
  \end{align}
  Since $\partial_y s^+$ and $\sigma^0(s) - \sigma^+(s)$ are bounded, $\partial_z\lambda^\pm\in L_1(K^\pm)$. 

  Now we turn to $E_{U}(f_t)$. Since $\bar{\phi}_t(U)=U$, Theorem~\ref{thm:fvf domain smooth} implies that
  $$\big|E_{U}(f_t) - E_{U}(f) - A_2 t\big|=\big|E_{K}(f_t) - E_{K}(f) - A_2 t\big|\le c \left(E_{U}(f) +\mu(U)\right) t^2$$
  for some $c=c(\psi, U, \|f\|_{L_\infty(U)})$, where
  $$A_2 = \int_{K} -\nabla_f^2 w_2 \cdot \nabla_{f}f + \frac{1}{2}  (\nabla_ff)^2\cdot \partial_z w_2 \ud \mu(v).$$
  We decompose $A_2$ as $A_2=A_2^++A_2^-$, where 
  \begin{equation}\label{eq:weak harmonic domain}
    A^\pm_2  =\int_{K^\pm} -\nabla_f^2w_2 \cdot \lambda^\pm \ud \mu + \int_{K^\pm} \frac{1}{2}  (\lambda^\pm)^2\cdot \partial_z w_2 \ud \mu =: Q_1^\pm + Q_2^\pm,
  \end{equation}
  and we will compute the $Q_i^{\pm}$'s as in the compactly supported case of Theorem~\ref{thm:fvf domain smooth}, plus a boundary term arising from the singularity.

  Consider $Q^+_1=\int_{K^+} -\nabla_f^2w_2 \cdot \lambda^+ \ud \mu$.
  We have
  $$\nabla_f w_2 =\nabla_f[(x\circ W)\cdot f + z\circ W]=\nabla_f[x\circ W]\cdot f + (x\circ W)\cdot \lambda^+ + \nabla_f[z\circ W]$$
  on $K^+$, so $\nabla_f w_2$ can be extended continuously to $\overline{K^+}$.  Letting $g=\nabla_{f}w_2$ and $h=\lambda^+$, we know that $g, h, \nabla_fg, \nabla_fh\in L_\infty(K^+)$.  By Lemma~\ref{lem:near sing} and \eqref{eq:pz lambda l1}, $\partial_zf,\partial_zh\in L_1(K^+)$. The smoothness of $w_2$ implies
  $$\partial_z g = \partial_z\partial_xw_2 - \partial_zf\cdot \partial_z w_2 - f\cdot \partial_z^2 w_2 \in L_1(K^+),$$
  so we can apply Corollary~\ref{cor:integration by parts} to show that
  \begin{align}\label{eq:q1 done}
    Q^+_1
    & = - \int_{K^+} \nabla_f w_2 \cdot \lambda^+ \cdot \partial_zf\ud\mu + \int_{K^+} \nabla_f \lambda^+ \cdot \nabla_fw_2 \ud \mu := Q^+_3 + Q^+_4.
  \end{align}

  Before we integrate by parts again, we restrict to a domain that avoids $\gamma$. Let $T_i\subset K^+$ be the closed subset of $K^+$ bounded by $Z^{\frac{1}{i}}\gamma$.  This has piecewise smooth boundary and $f$ is smooth on $T_i$. 
  Then $\nabla_f w_2, \lambda^+\in L_\infty(K^+)$, and $\partial_zf\in L_1(K)$ by Lemma~\ref{lem:near sing}, so
  $$Q^+_3 = \lim_{i\to \infty} - \int_{T_i} \nabla_f w_2 \cdot \lambda^+ \cdot \partial_zf\ud\mu.$$
  Let $\alpha_i$ parametrize $\partial T_i$ in the positive direction.  By Corollary~\ref{cor:integration by parts} with $g=w_2$, $h=\lambda^+ \cdot \partial_zf$, 
  \begin{multline}
\label{eq:Ti int}    - \int_{T_i} \nabla_f w_2
     \cdot \lambda^+ \cdot \partial_zf\ud\mu 
     = \int_{T_i}- w_2 \cdot \lambda^+ \cdot (\partial_zf)^2 + w_2 \cdot \nabla_f\left[\lambda^+ \cdot \partial_zf\right] \ud \mu+P_i\\
     = \int_{T_i} w_2 \cdot \lambda^+ \cdot \left(-(\partial_zf)^2 + \nabla_f[\partial_zf]\right)+w_2\cdot \partial_zf \cdot \nabla_f\lambda^+ \ud \mu+P_i
  \end{multline}
  where $P_i$ is the line integral $P_i=\int_{\partial T_i} -w_2 \cdot \lambda^+ \cdot \partial_zf \cdot (fX + Z)\cdot \ud \alpha_i$. As in the proof of Theorem~\ref{thm:fvf domain smooth}, $-(\partial_zf)^2 + \nabla_f[\partial_zf]= \partial_z\lambda^+$, so
  $$Q^+_3
  = \lim_{i\to \infty} \int_{T_i} w_2 \cdot \lambda^+ \cdot \partial_z \lambda^++ w_2\cdot \partial_zf \cdot \nabla_f\lambda^+ \ud \mu+P_i.$$
  This integrand is $L_1$, so in fact,
  \begin{equation}\label{eq:inner integral 2}
    Q^+_3 = \int_{K^+} w_2 \cdot \lambda^+ \cdot \partial_z \lambda^++ w_2\cdot \partial_zf \cdot \nabla_f\lambda^+ \ud \mu + \lim_{i\to \infty} P_i.
  \end{equation}
  
  Now we consider $\lim_{i\to \infty} P_i$. Let $\epsilon=i^{-1}$ and let $I$ be the domain of $\gamma$.  Since $w_2$ vanishes on $\partial T_i\setminus Z^{\epsilon}\gamma(I)$, it suffices to integrate $P_i$ over $Z^{\epsilon}\gamma(I)$.  Let $p=\gamma(t)$ and $q=\gamma(t) Z^{\epsilon}$. Since $\gamma$ is characteristic, $\gamma'(t) = X - f(p)Z$, so 
  \begin{align*}
    P_i
    &=\int_{I} - w_2(q) \cdot \lambda^+(q) \cdot \partial_zf(q)\cdot  (f(q) X + Z)\cdot \gamma'(t) \ud t\\
    &=\int_{I} - w_2(q) \cdot \lambda^+(q) \cdot \partial_zf(q)\cdot  (f(q) - f(p)) \ud t \\
    &= \int_{I} - w_2(q) \cdot \lambda^+(q) \cdot \left(\sigma^0(t)- \sigma^+(t)+O(\epsilon)\right) \ud t,
  \end{align*}
  using Lemma~\ref{lem:near sing} in the last step. Since $\lambda^+(q)=\sigma^+(t)$,
  \begin{equation}\label{eq:Pi done}
    \lim_i P_i= \int_{I} w_2(\gamma(t)) \cdot \sigma^+(t) \left(\sigma^+(t)- \sigma^0(t)\right) \ud t.
  \end{equation}

  Finally, $\partial_z\lambda^+\in L_1(K^+)$, so integrating $Q_2^+$ by parts gives
  $$Q^+_2=\int_{K^+}\frac{1}{2}  (\lambda^+)^2\cdot \partial_z w_2 \ud \mu
  = \int_{K^+} \frac{1}{2}\partial_z\left[w_2\cdot (\lambda^+)^2\right] \ud \mu - \int_{K^+} w_2\cdot \lambda^+ \cdot \partial_z\lambda^+ \ud \mu.$$
  Since $w_2$ vanishes on all of $\partial T_i$ except for its lower boundary $\gamma(I)$,
  \begin{align*}
    Q^+_2 
    &= -\int_{I} \frac{1}{2}  w_2(\gamma(s)) \cdot \lambda^+(\gamma(s))^2 \ud s - \int_{K^+} w_2\cdot \lambda^+ \cdot \partial_z\lambda^+ \ud \mu\\
    &= -\int_{I} \frac{1}{2}  w_2(\gamma(s)) \cdot \sigma^+(s)^2 \ud s - \int_{K^+} w_2\cdot \lambda^+ \cdot \partial_z\lambda^+ \ud \mu.
  \end{align*}

  Combining this computation with \eqref{eq:q1 done}, \eqref{eq:inner integral 2}, and \eqref{eq:Pi done} and canceling like terms, we find
  \begin{align*}
    A^+_2 & = Q^+_3+Q^+_4+Q^+_2 \\
        &= \int_{I} w_2(\gamma(s)) \cdot \left(\frac{1}{2} \sigma^+(s)^2 - \sigma^+(s)\sigma^0(s)\right) \ud s
          + \int_{K^+} (w_2\cdot \partial_zf + \nabla_fw_2)\cdot \nabla_f^2f \ud \mu.
  \end{align*}
  A similar calculation for $A^-_2$
  gives
  $$A^-_2 = \int_{I} w_2(\gamma(s)) \cdot \left(\frac{1}{2} \sigma^-(s)^2 - \sigma^-(s)\sigma^0(s)\right) \ud s
  + \int_{K^-} (w_2\cdot \partial_zf + \nabla_fw_2)\cdot \nabla_f^2f \ud \mu$$
  and
  $$A_2=A^+_2 + A^-_2 = \frac{1}{2} \int_{I} w_2(\gamma(s)) \cdot \delta(s)\ud s
  + \int_{K} (w_2\cdot \partial_zf + \nabla_fw_2)\cdot \nabla_f^2f \ud \mu,$$
  where $\delta(s)=(\sigma^+(s)-\sigma^0(s))^2 - (\sigma^0(s)-\sigma^-(s))^2$.
\end{proof}

\bibliographystyle{alphaurl}
\bibliography{hGraphs}
\end{document}